%% file: MaterialParametersViscoelasticMaterial_arXiv.tex
\documentclass[final,leqno]{siamltex}
\usepackage{chngcntr}

\usepackage[arrow, matrix, curve]{xy}
\usepackage[latin1]{inputenc} 	
\usepackage[T1]{fontenc}
\usepackage[english]{babel}
\usepackage{lmodern}
\usepackage{amsmath}
\usepackage{amssymb}  
\usepackage{mathrsfs}
\usepackage{euscript}
\usepackage{enumerate}
\usepackage{xspace}
\usepackage{hyperref}  
\usepackage{framed}
\usepackage{cite}
\usepackage{comment}
\usepackage{fancyhdr}
\usepackage{bbm}
\usepackage{float}
\usepackage{comment}
\usepackage{graphicx}
\usepackage{textcomp}
\usepackage{url}
\usepackage{srcltx} 
\usepackage{hyperref} 
\usepackage{subfig}

\DeclareFontShape{OMX}{cmex}{m}{n}{
  <-7.5> cmex7
  <7.5-8.5> cmex8
  <8.5-9.5> cmex9
  <9.5-> cmex10
}{}
\SetSymbolFont{largesymbols}{normal}{OMX}{cmex}{m}{n}
\SetSymbolFont{largesymbols}{bold}  {OMX}{cmex}{m}{n}

\hypersetup{draft=false, colorlinks=true, linkcolor=black, urlcolor=blue, citecolor=black}

\newtheorem{example}[theorem]{Example}

\author{Rebecca Rothermel\footnote[1]{Department of Mathematics, Saarland University Saarbr\"ucken, Germany} \and Wladimir Panfilenko\footnotemark[1] \and Prateek Sharma\footnote[4]{Chair of Applied Mechanics, Saarland University Saarbr\"ucken, Germany} \and Anne Wald\footnote[2]{Institute for Numerical and Applied Mathematics, University of G\"ottingen, Germany} \and Thomas Schuster\footnotemark[1]\;\;\footnote[3]{{\tt email: thomas.schuster@num.uni-sb.de}, corresponding author} \and Anne Jung\footnotemark[4] \and Stefan Diebels\footnotemark[4]}

\title{A method for determining the parameters in a rheological model for viscoelastic materials by minimizing Tikhonov functionals}

\begin{document}

\maketitle

\begin{abstract}
Mathematical models describing the behavior of viscoelastic materials are often based on evolution equations that measure the change in stress depending on its material parameters such as stiffness, viscosity or relaxation time.
In this article, we introduce a Maxwell-based rheological model, define the associated forward operator and the inverse problem in order to determine the number of Maxwell elements and the material parameters of the underlying viscoelastic material. We perform a relaxation experiment by applying a strain to the material and measure the generated stress.
Since the measured data varies with the number of Maxwell elements, the forward operator of the underlying inverse problem depends on parts of the solution. By introducing assumptions on the relaxation times, we propose a clustering algorithm to resolve this problem. We provide the calculations that are necessary for the minimization process and conclude with numerical results by investigating unperturbed as well as noisy data. We present different reconstruction approaches based on minimizing a least squares functional. Furthermore, we look at individual stress components to analyze different displacement rates. Finally, we study reconstructions with shortened data sets to obtain assertions on how long experiments have to be performed to identify conclusive material parameters. 
\end{abstract}



\section{Introduction}
\label{sec:intro}


Simulations are used to predict the behavior of a material without conducting expensive experiments. In order to achieve accurate results from a simulation, the physical behavior as well as the properties of the material should be captured by the simulation model. The behavior of the material can be interpreted by the qualitative analysis of an experimental result obtained from material testing. However, the properties of the material need to be quantitatively determined through comparison of the simulation and experimental results. With variation and adjustment of the simulation results, the parameters of the model are optimised to fit the experimental data. Thus, parameter identification in materials science is an important field of application for inverse problems with many different purposes. 
Among them are the identification of the stored energy function of hyperelastic materials \cite{hartmann2003polyconvexity, woestehoff2015uniqueness, binder2015defect, seydel2017linearization, seydel2017identifying, klein2019sequential}, the surface enthalpy dependent heat fluxes of steel plates \cite{rothermel2020solving, rothermel2019parameter}, inverse scattering problems 
\cite{COLTON;KRESS:98} or the refractive index through terahertz tomography \cite{awts18}. All these inverse problems are ill-posed in the sense that even small perturbations in the measured data cause severe artifacts in the solution. This is why the application of regularization methods is necessary. A standard approach is to use the residual between simulations and measurements as a
data fitting term, see \cite{Chen2013, Heinze2018}. Further authors conduct a sensitivity analysis to determine the possible deviations in the parameters due to noise \cite{Gavrus1996, Hartmann2001, hartmann2001numerical, Mahnken1996}.

Viscoelastic material behavior is exhibited by various polymeric materials such as adhesives, elastomers and rubber. The modeling of these materials requires parameters that define the relaxation time spectrum, to model the loading rate dependent behavior \cite{Scheffer2013, bonfanti2020fractional, jalocha2020payne, sharma2020}. The reconstruction of these parameters do not only depend on the precision, but also on the duration of the experiments. The numerical implementation of a continuous relaxation time spectrum with discrete parameters represents a further challenge. Therefore, the parameter identification of viscoelastic materials is in great demand and is used in many fields.
Results for the identification of material parameters of viscoelastic structures with a linear model can also be found in \cite{diebels2018identifying, FERNANDA;ET;AL:11}. High-order discontinuous Galerkin methods for anisotropic viscoelastic wave equations are found in \cite{SHUKLA;ET;AL:11}. There, the authors use a constitutive equation with memory variables.
An overview of different methods for solving inverse problems and deeper insights to regularization theory can be found in the standard textbooks \cite{engl1996regularization, louis2013inverse, rieder2013keine, KIRSCH:11}. For nonlinear problems, such as the one presented in this article, these methods have to be adapted correspondingly (see \cite{kaltenbacher2008iterative, rieder2013keine, wald2017sequential, wald2018fast}).

\textit{Outline.} In Section \ref{sec:material} we describe the rheological model for viscoelastic materials on which the investigations in this article are based. Assuming that the number of Maxwell elements are given the constitutive model can be solved analytically, which is done in Section \ref{sec: Analytische Loesung}. Relying on this model the simulation of data is performed in Section \ref{sec:data-simulation}. The task of computing the model parameter from measured stresses at certain time instances is an inverse problem. In contrast to standard settings for inverse problems, the forward operator depends in this case on parts of the solution, namely the number of maxwell elements. In Section \ref{sec: Cluster} we develop a clustering algorithm to overcome this difficulty. In Section \ref{sec:NumericalResults} we perform a series of numerical experiments using exact as well as noisy data. Since the inverse problem is ill-posed we introduce regularization methods by minimizing corresponding Tikhonov functionals and prove that this stabilizes the solution process. Furthermore we demonstrate the effect of shortening the data series and finally obtain assertions on its influence to the reconstruction result. A concluding section finishes the article.


\section{A mathematical model for material behavior} 
\label{sec:material}


We consider relaxation experiments in which a strain is applied to a material. 
For a given displacement rate $\eta$ and maximum strain value $\bar{\varepsilon}$ we write the strain in a time interval $t \in [0,T]$ as 
\begin{align}
\varepsilon(t)= \left\{\begin{array}{ll} \eta \cdot t, &0 \leq t\leq \frac{\bar{\varepsilon}}{\eta} \\
         \bar{\varepsilon}, &  \frac{\bar{\varepsilon}}{\eta} <t \leq T. \end{array}\right. 
        \label{eqn:strain}
\end{align}
The function is illustrated in Figure \ref{abb:strain_exp} and describes the following procedure. The material is stretched at a displacement rate $\eta$ until we reach a maximum strain value $\bar{\varepsilon}$. This happens accordingly at time $t=\frac{\bar{\varepsilon}}{\eta}$. Afterwards the applied strain is kept constant. 

\begin{figure}[h]
\centering
\def\svgwidth{.6\linewidth}
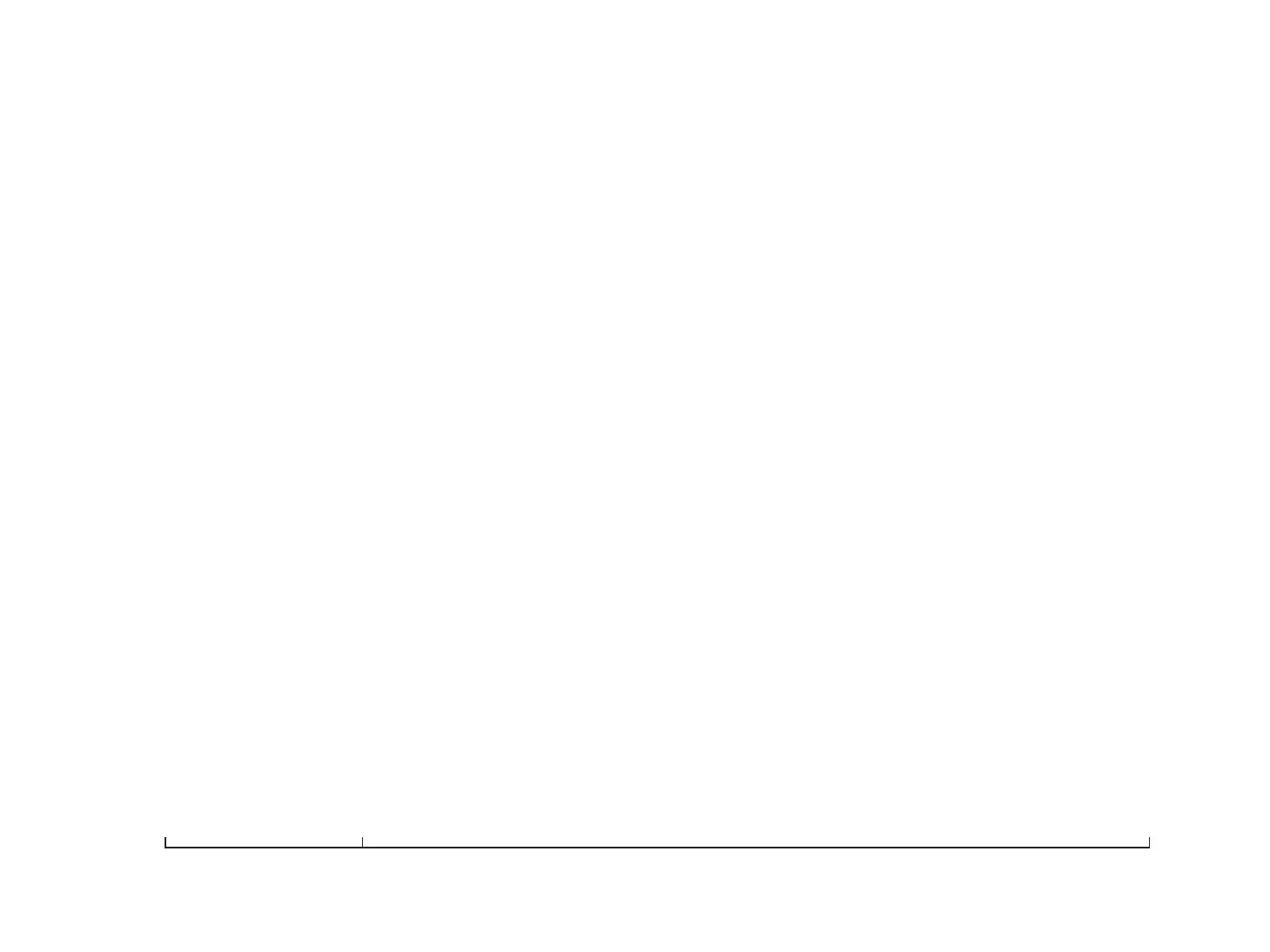
\caption{Strain-time curve $\varepsilon$ with displacement rate $\eta$ and maximum strain value $\bar{\varepsilon}$ }
\label{abb:strain_exp}
\end{figure}

We describe the stress in the material by a rheological multi-parameter model combining several Maxwell elements. Two opposing properties are essentially determining the time evolution of a material under strain, elasticity and viscosity. Elasticity is modeled by a spring, whereas viscosity is modeled by a damper. The series composition of a spring and a damper yields a Maxwell element.
We use one of the most common combinations to represent the material, a parallel combination of a spring with different Maxwell elements \cite{Goldschmidt2015b, Johlitz2007, Scheffer2013}. A parallel combination of these elements ensures that the entire relaxation spectrum of the material can be represented by the model.

\begin{figure}[h]
\centering
\def\svgwidth{.6\linewidth}
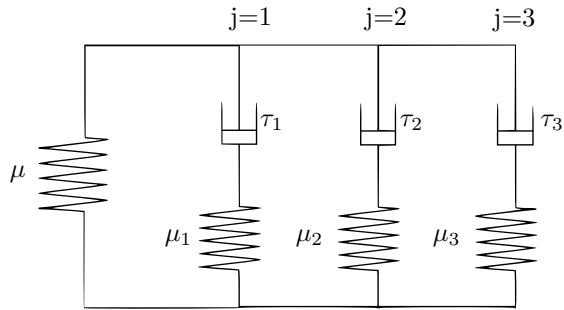
\caption{Rheological Model with three Maxwell elements}
\label{abb:rheologicalModel}
\end{figure}

The number of Maxwell elements can extend to any number $n$ with a relaxation time $\tau_j$ in each of the dampers and a stiffness of the spring $\mu_j$ in each of the Maxwell elements. In Figure \ref{abb:rheologicalModel} we see an illustration with three Maxwell elements. The deformation across the damper and across the spring in a Maxwell element changes according to the relaxation time of the damper. The Maxwell springs reach the original undeformed position and do not produce any stress after fully relaxing. Therefore, with a single Maxwell element the model can only simulate one stiffness for the entire duration of relaxation of the material. However, viscoelastic materials show different stiffness at different relaxation times \cite{sharma2020, Goldschmidt2015b}. The entire relaxation spectrum can be divided into different regions such as the flow region, the entanglement region, the transition region and the glassy region\cite{baurngaertel1992}. The arrangement, the length and the entanglement of the polymeric chains on the molecular level explains the varying stiffness during the relaxation spectrum. A shorter chain relaxes faster than a longer chain and hence after its relaxation does not contribute to the stiffness of the material \cite{berry1968}. Moreover, if all the Maxwell elements relax to zero stresses then the equilibrium position is attained and the single spring with the stiffness $\mu$ represents the basic stiffness of the material, which ensures that the material has a stiffness in its unperturbed position and does not behave like a fluid.
The deformation of the material represented by the strain value $\varepsilon$ is divided into an elastic component $\varepsilon_j^{\,e}$ and an inelastic component $\varepsilon_j^{\,i}$ in each of the Maxwell elements. The elastic component corresponds to the spring and the non-elastic component to the damper. As the stretching of the damper is dependent on its relaxation time an evolution equation based on the energy conservation is used to model the change in the inelastic strain with time \cite{johlitz2013, reese1998}. For small deformations the evolution equation is given by
\begin{equation}
\dot{\varepsilon}_j^{\,i}(t) = \frac{\varepsilon(t) - \varepsilon^{\,i}_j(t)}{\tau_j/2},
\label{eqn:evolution} \
\end{equation} 
where $\dot{\varepsilon}_j^{\,i}$  represents the time derivative. 
The total stress produced in the system is then given by the sum of stresses induced in each of the springs. Assuming a linear elastic behavior of the springs. The stress is given by
\begin{equation}
\sigma(t) = \mu \, \varepsilon(t) + \sum\limits_{j=1}^{n}\, \mu_j (\varepsilon(t) - \varepsilon_j^{\,i}(t)).
\label{eqn:stressComputation}
\end{equation}
With this information we have now all ingredients available to construct the forward operator as 
\begin{align}\label{eqn: Forward operator}
F_n: \mathbb{R}_{\geq 0}^{2n+1} \times \mathbb{N} &\to L^2([0,T])\\
\left(\mu, \mu_1, \dots, \mu_n, \tau_1, \dots \tau_n, n \right) &\mapsto \sigma,
\end{align}
where $n$ is the number of Maxwell elements and $\sigma$ is the stress defined by equation \eqref{eqn:stressComputation}.
We note that the forward operator depends on parts of the solution at this point since we require the unknown number of Maxwell elements $n$ to calculate the stress. We will solve this problem in Section \ref{sec: Cluster}. 

At this point, the inelastic stresses are included as a constraint in the formulation of the inverse problem. In the following section we will see that these constraints can be directly included to obtain an explicit formulation of the forward operator.


\section{Analytical solution of the forward problem} 
\label{sec: Analytische Loesung}


For the solution of the forward problem, we first want to focus on the calculation of the inelastic component of the strain value $\varepsilon_j^i$, which is not explicitly given as a so-called internal variable of the system.
An evolution equation such as \eqref{eqn:evolution} is often solved by numerical methods like the Crank-Nicolson method. However, these methods have the disadvantage of approximation errors. Since we are dealing with an ill-posed inverse problem, it is desirable to avoid such errors in the calculation of the forward problem. 
To this end, we discuss the analytical solution of the evolution equation and, using this, the forward problem in the next section. 

Apart from the mechanical application \cite{reese1998}, the evolution equation \eqref{eqn:evolution} is also used in other applications like Magnetic Particle Imaging  \cite{croft2012relaxation} to model relaxation. Its solution can be formulated analytically as
\begin{align*}
\varepsilon_j^i(t)=\int \limits_0^t \varepsilon(\tilde{t}) \frac{2}{\tau_j} \exp\left( -2 \frac{t-\tilde{t}}{\tau_j} \right) d\tilde{t}.
 \end{align*}

By inserting the piecewise linear strain \eqref{eqn:strain} we obtain the following result:
 
\begin{proposition}

The inelastic strain $\varepsilon_j^i$ of the $j$-th Maxwell element with the corresponding relaxation time $\tau_j$ at time $t\in [0,T]$ is given by
\begin{align}
\varepsilon_j^i(t)= \left\{\begin{array}{ll}  \frac{\tau_j}{2}\eta \exp\left( -\frac{2}{\tau_j} t \right)+\eta t -\frac{\tau_j}{2}\eta , &0 \leq t\leq \frac{\bar{\varepsilon}}{\eta} \\
        \frac{\tau_j \eta}{2}\exp\left( -\frac{2}{\tau_j} t \right) \left[ 1- \exp\left( \frac{2\bar{\varepsilon}}{\tau_j \eta} \right)  \right] +\bar{\varepsilon}, &  \frac{\bar{\varepsilon}}{\eta} <t \leq T. \end{array}\right. 
        \label{eqn:strain_inelastic}
\end{align}

\end{proposition}

\begin{proof}
We first consider $0 \leq t\leq \frac{\bar{\varepsilon}}{\eta}$ and can thus use $\varepsilon(t)=\eta t$,
\begin{align*}
  \varepsilon_j^i(t)&= \frac{2}{\tau_j} \eta \exp\left(-\frac{2}{\tau_j} t\right)  \int \limits_0^t \tilde{t} \exp\left(  \frac{2}{\tau_j} \tilde{t} \right) d\tilde{t}\\
 &=  \frac{2}{\tau_j} \eta \exp\left(-\frac{2}{\tau_j} t\right) \frac{\tau_j^2}{4} \left[ \exp\left( \frac{2}{\tau_j} t \right) \left( \frac{2}{\tau_j} t - 1 \right) +1 \right]    , 
\end{align*}
where we used $\int \limits_0^t \tilde{t} \exp\left(a\tilde{t}\right) d\tilde{t} = \frac{1}{a^2} \left(\exp\left(at \right) (at-1)  +1 \right)$ for $a \in \mathbb{R}$ in the last step. The remainder is only a matter of transformation, 
\begin{align*}
  \varepsilon_j^i(t) &=
   \frac{\tau_j}{2} \eta \exp\left( - \frac{2}{\tau_j} t \right) \frac{2}{\tau_j}  t\exp\left( \frac{2}{\tau_j} t \right)\\
   & \ \ \ -\frac{\tau_j}{2} \eta \exp \left( -\frac{2}{\tau_j}  t\right)\exp \left( \frac{2}{\tau_j}  t\right)\\
   & \ \ \ + \frac{\tau_j}{2} \eta \exp\left(- \frac{2}{\tau_j} t \right)\\
   &= \frac{\tau_j}{2}\eta \exp\left( -\frac{2}{\tau_j} t \right)+\eta t -\frac{\tau_j}{2}\eta. 
\end{align*}
Next we examine the case $\frac{\bar{\varepsilon}}{\eta} <t \leq T$, where we have to split the integral in two parts to be able to use the corresponding value for $\varepsilon(t)$,

\begin{align*}
  \varepsilon_j^i(t) = 
   \frac{2}{\tau_j} \exp\left(-\frac{2}{\tau_j} t\right)  \left[ \eta \int \limits_0^{\frac{\bar{\varepsilon}}{\eta}} \tilde{t} \exp\left(  \frac{2}{\tau_j} \tilde{t} \right) d\tilde{t} + \bar{\varepsilon} \int \limits_{\frac{\bar{\varepsilon}}{\eta}}^t \exp\left(  \frac{2}{\tau_j} \tilde{t} \right) d\tilde{t} \right].
\end{align*}
For each integral we then get 
\begin{align*}
  \int \limits_0^{\frac{\bar{\varepsilon}}{\eta}} \tilde{t} \exp\left(  \frac{2}{\tau_j} \tilde{t} \right) d\tilde{t}= \frac{\tau_j^2}{4} \left[ \exp\left( \frac{2\bar{\varepsilon}}{\tau_j \eta} \right) \left( \frac{2\bar{\varepsilon}}{\tau_j \eta} -1 \right) +1  \right],\\
  \int \limits_{\frac{\bar{\varepsilon}}{\eta}}^t \exp\left(  \frac{2}{\tau_j} \tilde{t} \right) d\tilde{t}= \frac{\tau_j}{2} \left[ \exp\left( \frac{2}{\tau_j} t \right) -\exp \left( \frac{2\bar{\varepsilon}}{\tau_j \eta} \right)   \right]. 
\end{align*}
Inserting these results yields 
\begin{align*}
  \varepsilon_j^i(t) 
   &=\frac{2}{\tau_j} \exp\left(-\frac{2}{\tau_j} t\right)
   \left[ \eta \frac{\tau_j^2}{4} \left[ \exp\left( \frac{2\bar{\varepsilon}}{\tau_j \eta} \right) \left( \frac{2\bar{\varepsilon}}{\tau_j \eta} -1 \right) +1  \right] 
   + \bar{\varepsilon} \frac{\tau_j}{2} \left[ \exp\left( \frac{2}{\tau_j} t \right) -\exp \left( \frac{2\bar{\varepsilon}}{\tau_j \eta} \right)   \right] \right]\\
  & = \frac{2}{\tau_j} \exp\left(-\frac{2}{\tau_j} t\right) \\   
   &\qquad \left[ \eta \frac{\tau_j^2}{4}  \frac{2\bar{\varepsilon}}{\tau_j \eta}  \exp\left( \frac{2\bar{\varepsilon}}{\tau_j \eta} \right) - \eta \frac{\tau_j^2}{4} \exp\left( \frac{2\bar{\varepsilon}}{\tau_j \eta} \right) + \eta \frac{\tau_j^2}{4} + \bar{\varepsilon}\frac{\tau_j}{2} \exp\left( \frac{2}{\tau_j} t \right) - \bar{\varepsilon} \frac{\tau_j}{2} \exp \left( \frac{2\bar{\varepsilon}}{\tau_j \eta} \right) \right] \\
   &=\exp\left( -\frac{2}{\tau_j t} \right) \left[ \exp\left( \frac{2\bar{\varepsilon}}{\tau_j \eta} \right) \left( \bar{\varepsilon} -\frac{\tau_j \eta}{2} -\bar{\varepsilon} \right) +\frac{\tau_j \eta}{2} + \exp\left( \frac{2}{\tau_j t} \right) \bar{\varepsilon} \right] \\
  &= -\frac{\tau_j \eta}{2} \exp \left( -\frac{2}{\tau_j} t + \frac{2 \bar{\varepsilon}}{\tau_j \eta} \right) +\frac{\tau_j \eta}{2} \exp \left( -\frac{2}{\tau_j} t  \right) + \bar{\varepsilon}  \\
  &=  \frac{\tau_j \eta}{2}\exp\left( -\frac{2}{\tau_j} t \right) \left[ 1- \exp\left( \frac{2\bar{\varepsilon}}{\tau_j \eta} \right)  \right] +\bar{\varepsilon}.
\end{align*}
\end{proof}

We use this result to obtain the stress responses of the individual Maxwell elements. This is also useful for an analysis of different displacement rates $\eta$ in Section \ref{sec:different strain rates}. 

\begin{proposition}\label{P-exact-sigma}
The stress $\sigma_j$ of the $j$-th Maxwell element with $j>1$ and the corresponding stiffness $\mu_j$ and relaxation time $\tau_j$ at time $t\in [0,T]$ is given as 
\begin{align}
\sigma_j(t)= \left\{\begin{array}{ll}  \frac{\mu_j \tau_j \eta}{2} \left( 1- \exp \left( -\frac{2}{\tau_j} t \right)  \right)  , &0 \leq t\leq \frac{\bar{\varepsilon}}{\eta} \\
       - \frac{\mu_j \tau_j \eta}{2}  \left( 1- \exp \left( -\frac{2 \bar{\varepsilon} }{\tau_j \eta }  \right)  \right) \exp \left( -\frac{2}{\tau_j} t \right), &  \frac{\bar{\varepsilon}}{\eta} <t \leq T. \end{array}\right. 
        \label{eqn:stress_j}
\end{align}
The stress of the single spring, denoted by $\sigma_{j=0}$, can be specified with the corresponding stiffness $\mu$ as
\begin{align*}
    \sigma_0(t)= \left\{\begin{array}{ll} \mu \eta t  , &0 \leq t\leq \frac{\bar{\varepsilon}}{\eta} \\
       \mu \bar{\varepsilon}, &  \frac{\bar{\varepsilon}}{\eta} <t \leq T. \end{array}\right. 
\end{align*}

According to these calculations, the total stress can be written as the sum of the stresses of the single spring and the Maxwell elements, i.e.
\begin{align}
    \sigma (t)= \sum \limits_{j=0}^n \sigma_j(t) = \left\{\begin{array}{ll} \mu \eta t + \sum \limits_{j=1}^n \frac{\mu_j \tau_j \eta}{2} \left( 1- \exp\left( -\frac{2}{\tau_j} t \right)  \right) , &0 \leq t\leq \frac{\bar{\varepsilon}}{\eta} \\
      \mu \bar{\varepsilon} - \sum \limits_{j=1}^n \frac{\mu_j \tau_j \eta}{2}  \left( 1- \exp \left( -\frac{2 \bar{\varepsilon} }{\tau_j \eta }  \right)  \right) \exp \left( -\frac{2}{\tau_j} t \right) , &  \frac{\bar{\varepsilon}}{\eta} <t \leq T. \end{array}\right. 
      \label{eq: sigmaTotal}
\end{align}

\end{proposition}

\begin{proof}
We start with the stress of the single spring. This is calculated by
$\sigma_0(t)=\mu \varepsilon(t)$. Inserting \eqref{eqn:strain} for the different time intervals gives us the stress as shown above.
The stress of the Maxwell elements is calculated by $\sigma_j(t)=\mu_j(\varepsilon(t)- \varepsilon_j^i(t) )$ as given in equation \eqref{eqn:stressComputation}. 
By considering the different time intervals we can insert both, the strain $\varepsilon(t)$ and the inelastic strain $\varepsilon_j^i(t)$ of the Maxwell element.
Hence, for $0 \leq t\leq \frac{\bar{\varepsilon}}{\eta}$ we obtain
\begin{align*}
    \sigma_j(t)=\mu_j \left( \eta t -\frac{\tau_j}{2}\eta \exp\left( -\frac{2}{\tau_j} t \right)-\eta t +\frac{\tau_j}{2}\eta \right) 
    = \frac{\mu_j \tau_j \eta}{2} \left( 1-\exp\left( -\frac{2}{\tau_j} t \right) \right).
\end{align*}
Similarly, for $\frac{\bar{\varepsilon}}{\eta} <t \leq T$, we obtain 
\begin{align*}
    \sigma_j(t)&=\mu_j \left( \bar{\varepsilon} - \frac{\tau_j \eta}{2}\exp\left( -\frac{2}{\tau_j} t \right) \left[ 1- \exp\left( \frac{2\bar{\varepsilon}}{\tau_j \eta} \right)  \right] -\bar{\varepsilon} \right) \\
    &=  - \frac{\mu_j \tau_j \eta}{2}  \left( 1- \exp \left( -\frac{2 \bar{\varepsilon} }{\tau_j \eta }  \right)  \right) \exp \left( -\frac{2}{\tau_j} t \right).
\end{align*}
The representation of the total stress \eqref{eq: sigmaTotal} follows directly by inserting \eqref{eqn:stress_j} in \eqref{eqn:stressComputation}.
\end{proof}

Since the stress can be expressed in terms of the material parameters, we are able to formulate the forward operator \eqref{eqn: Forward operator}. It is still solution-dependent, because we still need the number of Maxwell elements to calculate the stress, but this is unknown. 
Before we address this topic, we will briefly discuss the simulation of data, where the above results will be helpful. 


\section{Data simulation}
\label{sec:data-simulation}


We simulate a relaxation experiment where a strain is applied to a material. This strain is completely determined by the parameters $\eta$ (displacement rate) and $\bar{\varepsilon}$ (maximum strain).  So, if we choose these two parameters and a time period $[0,T]$, we have completely determined the strain using the representation \eqref{eqn:strain}. This function can then be supposed to be known for all $t\in [0,T]$. 
In Figure \ref{abb:strains}, we see strains for three different displacement rates $\eta$ and the maximum strain $\bar{\varepsilon}=20\,\%$. The fastest displacement rate is $10$ mm/s  with a maximum strain being attained after $2$ seconds. For the slowest displacement rate of $1$ mm/s, the $20 \%$ strain is only achieved after $20$ seconds. We chose $T=100$ seconds. 

\begin{figure}[h]
\centering
\def\svgwidth{.7\linewidth}
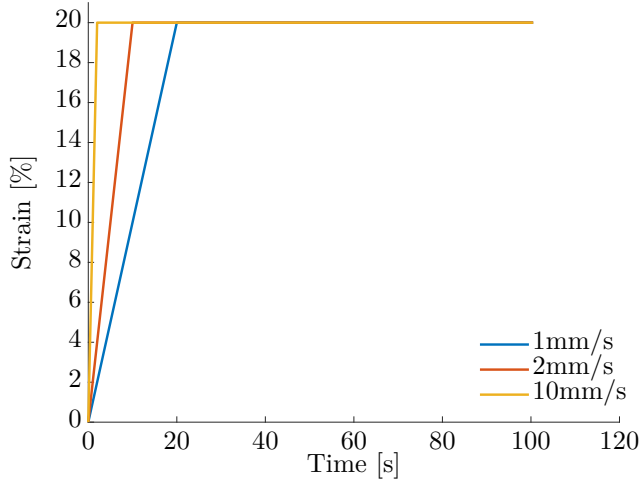
\caption{Strain-time curves for different displacement rates $\eta$}
\label{abb:strains}
\end{figure}

For the simulation of the stress function we then choose the number of Maxwell elements $n$ and the material parameters $\left(\mu, \mu_1, \dots, \mu_n, \tau_1, \dots \tau_n \right) $. Thus, we can solve the forward problem with equation \eqref{eq: sigmaTotal} obtaining the stress function $\sigma$.
In Figure \ref{abb:stress} we see the stress functions corresponding to the strains with different displacement rates $\eta$ of Figure \ref{abb:strains}. Here, we have selected one spring and three Maxwell elements with the material parameters $\mu=10$, $(\mu_1,\tau_1)= (4, 0.2)$, $(\mu_2,\tau_2)= (7, 3.7)$ and $(\mu_3,\tau_3)= (1, 25)$. The stiffnesses $\mu$ and $\mu_j$ have the unit MPa and the unit of the relaxation times of the different dampers are seconds. The relaxation times were selected in three different decades to simulate a relaxation time spectrum, whereas the stiffness for each of the elements were selected randomly to ensure no correlation between the relaxation time and the stiffness.
Since the Maxwell elements are connected in parallel in the rheological model as shown in Figure \ref{abb:rheologicalModel}, the order of the Maxwell elements is not important. This is also evident in the forward model \eqref{eq: sigmaTotal}, since we simply have a sum of the stress contributions of the individual Maxwell elements and the single spring. 
So we stick to the convention to always number them according to the order of the relaxation times.

\begin{figure}[h]
\centering
\def\svgwidth{.7\linewidth}
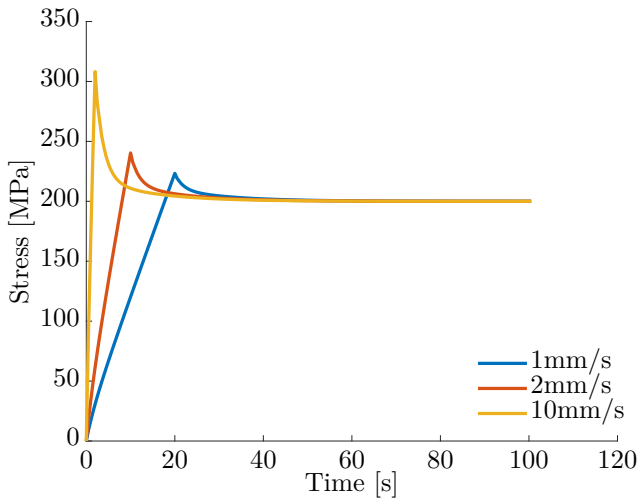
\caption{Synthetic stress-time data produced by a model consisting of a spring combined with three Maxwell elements for three different strain rates}
\label{abb:stress}
\end{figure}

For the calculation of $\sigma$ we have to discretize with respect to time.  We choose $m+1$ as the number of time steps $t_i=i \cdot \frac{T}{m}$ for $i=0, \dots, m$ in the interval $[0,T]$ and thus also have a discretized representation of the stress function with $\left( \sigma(t_0), \dots, \sigma(t_m) \right) \in \mathbb{R}^{m+1}$. To avoid an inverse crime, one should choose different discretizations for the direct and inverse problem. But since we can analytically compute the forward problem depending on the number of Maxwell elements $n$, this problem is not apparent. In Section \ref{sec:NumericalResults} we discuss results of experiments and show various scenarios that lead to reconstructions of different accuracy concerning the material parameters.

For any ill-posed inverse problem, reconstruction is considerably more difficult in case of noisy data $\sigma^\delta$. In this regard, we will also see different results in Section \ref{sec:NumericalResults}.
We simulate noise by adding scaled standard normally distributed random numbers to the discretized stress vector such that $\lVert \sigma - \sigma^\delta\rVert < \delta$ with noise level $\delta >0$. A noise level $\delta=0$ corresponds to unperturbed data.


\section{Solving the inverse problem} 
\label{sec: Cluster}


In this section we develop a clustering algorithm, which allows us to overcome the fact that the forward operator depends on the unknown number of Maxwell elements. 
Currently we cannot compute the forward operator without a fixed number of Maxwell elements.
Together with the ambiguity of our forward operator, this contributes to the ill-posedness of the problem. To illustrate the mentioned ambiguity, we will consider the following example.

\begin{example}
For a material described by two Maxwell elements with relaxation times $\tau_1=\tau_2$ and stiffnesses $\mu_1,\mu_2$ we have $\varepsilon_1^i=\varepsilon_2^i$. 
We furthermore consider another material described by the same basic stiffness but only one Maxwell element with relaxation time $\tilde{\tau_1}=\tau_1=\tau_2$ and stiffness $\tilde{\mu_1}$. 
If additionally $\tilde{\mu_1}=\mu_1+\mu_2$ applies, there is no possibility to distinguish these two materials by means of the stress created by our forward operator. 
The two stress-time curves are exactly the same with no possibility to reconstruct the parameters unambiguously. 
\end{example}

This shows, how to construct different tuples $(\mu, \mu_1, \dots, \mu_n, \tau_1,\ldots, \tau_n, n)$, such that 
$$F_n\left(\mu, \mu_1, \dots, \mu_n, \tau_1, \dots \tau_n, n \right)=\sigma$$
for a specific $\sigma \in L_2([0,T])$. For this reason we make further requirements. A common assumption is that the occurring relaxation times lie in different decades \cite{goldschmidt2015}. For instance, if $\tau_l \in [10,100]$ for $l \in \{ 1,\dots , N \}$, then  $\tau_j \notin [10,100]$ for all $j\in \{1, \dots,l-1,l+1, \dots N \}$.
This information will help us to solve the ill-posed problem.

We set a maximum number of Maxwell elements $N$, such that the unknown actual number of Maxwell elements $n$ is smaller or equal to $N$. Then, we reconstruct the desired parameters $(\mu, \mu_1, \dots, \mu_N, \tau_1, \dots \tau_N)$ with a minimization algorithm, obtaining the stiffness parameter for the single spring as well as stiffnesses and relaxation times for $N$ Maxwell elements. 
Of course, at this point there could be too many parameters, because the number of Maxwell elements could be too large. Nevertheless these values are helpful.  
They do not fulfill the requirement of having the relaxation times in different decades yet. Therefore, we apply a clustering algorithm to these parameters, which clusters them according to this requirement. 
That means our reconstruction consists of two parts:
\begin{itemize}
\item minimization algorithm with $N$ Maxwell elements,
\item clustering algorithm which reduces $N$ to $n$ Maxwell elements.
\end{itemize}
We describe this in more detail in the following.

First we will explain the minimization process. 
At this point we are faced with the following problem: We have applied a known strain to a material, measured the generated stress and want to deduce the parameters of the material. 
To this end we minimize the residual
\begin{align} \label{eqn: Residuum}
R= \lVert F_N\left(\mu, \mu_1, \dots, \mu_N, \tau_1, \dots \tau_N,N\right) - \sigma \rVert_2^2.
\end{align}  
This minimization process is done with $N$ Maxwell elements. Since the problem is nonlinear with respect to the relaxation times, we use an iterative solution method with several starting values for the stiffnesses and relaxation times 
$(\mu^{(0)}, \mu_1^{(0)}, \dots, \mu_N^{(0)}, \tau_1^{(0)}, \dots \tau_N^{(0)})$. 
The starting values are distributed over the possible range of parameter values. The usage of several starting values is necessary because of the nonlinearity of the problem which in general causes the existence of several local minima. 
For each of those starting values we then apply the forward operator to these parameters and calculate \eqref{eqn: Residuum}. This gives us an estimate of how well our initial values match the desired parameters. We change the starting values and the process starts again in a new step. 
This way the residual
\begin{align*}
R^{(k)}= \Big\lVert F_N\big(\mu^{(k)}, \mu_1^{(k)}, \dots, \mu_N^{(k)}, \tau_1^{(k)}, \dots \tau_N^{(k)} ,N\big) - \sigma \Big\rVert_2^2
\end{align*} 
with the iteration index $k=0,1,\dots, K$ is recalculated in each step, which allows us to evaluate the changes in the parameters. For this process we use the function 'lsqnonlin' and 'MultiStart' provided in the MATLAB Optimization Toolbox \cite{mathworks2005matlab}.

Next we look at the clustering algorithm mentioned before.
For a decade $[10^k, 10^{k+1}]$ with any $k\in \mathbb{N}$ we consider the index set 
\begin{align*}
J_k:=\lbrace j \in \lbrace 1,\dots, N \rbrace: \tau_j \in [10^k, 10^{k+1}]   \rbrace.
\end{align*}
Of course, depending on the application, other intervals instead of decades can also be of interest. All pairs $(\mu_j,\tau_j)$, $j\in J_k$, must be assigned to one Maxwell element. If we cluster the relaxation times, and thus the Maxwell elements, into index sets, the number of non-empty index sets gives us the actual number $n$ of Maxwell elements in the material. There are two ways to continue: 
\begin{itemize}
    \item[1)] Start the minimization process again with the actual number of Maxwell elements $n$. 
    \item[2)] Calculate suitable parameters $(\tilde{\mu}_k, \tilde{\tau}_k)$ with $\tilde{\tau}_k \in [10^k, 10^{k+1}]$ from the already reconstructed parameters through  
    \begin{align}
    \tilde{\mu}_k:= \sum_{j\in J_k} \mu_j, \qquad \tilde{\tau}_k:= \sum_{j\in J_k} \frac{\mu_j}{\tilde{\mu}_k} \tau_j.
    \label{eq:cluster calculation}
    \end{align}
\end{itemize}
Point 2) has of course the advantage that there is no need for a new elaborate minimization process. The drawback is that there is no analytic way to combine the parameters, so that the exact same stress curve is obtained by applying the forward operator. But, since we deal with a numerical approximation of the parameters, the calculation of the new parameters from already reconstructed ones can result in an accidental improvement. 

\begin{example} \label{example: Cluster}
We look at different cases and how they are treated by the clustering algorithm.
\begin{enumerate}
    \item Let $ J_k=\lbrace 1,3,5 \in \lbrace 1,\dots, N \rbrace: \tau_j \in [10^k, 10^{k+1}] \rbrace $ with $\tau_1=\tau_3=\tau_5$ and $N=6$. It is obvious that this should actually be a single Maxwell element.
Consider the forward operator with 
\begin{align*}
    \sigma(t) = \mu \, \varepsilon(t) + \sum\limits_{j=1}^{N}\, \mu_j \big(\varepsilon(t) - \varepsilon_j^{\,i}(t)\big).
\end{align*}
Since $\varepsilon_j^{i}$ varies only in $\tau_j$ for different Maxwell elements, it follows that $\varepsilon_1^{i}=\varepsilon_3^{i}=\varepsilon_5^{i}$ leading to
\begin{align}\label{forward_op}
    \sigma(t) = \mu \, \varepsilon(t) + (\mu_1+\mu_3+\mu_5)(\varepsilon(t) - \varepsilon_1^{i}(t))+ \sum\limits_{j=2,4,6}\, \mu_j (\varepsilon(t) - \varepsilon_j^{\,i}(t)).
\end{align}
Therefore, it is reasonable to choose the new values as
$$(\tilde{\mu}_k, \tilde{\tau}_k)= (\mu_1+\mu_3+\mu_5, \tau_1).$$
This is done by the clustering algorithm as
\begin{align*}
\tilde{\mu}_k&= \sum_{j\in J_k} \mu_j= \mu_1+\mu_3+\mu_5,\\
\tilde{\tau}_k&= 
\sum_{j\in J_k} \frac{\mu_j}{\tilde{\mu}_k} \tau_j= \frac{\mu_1+\mu_3+\mu_5}{\mu_1+\mu_3+\mu_5} \tau_1=\tau_1 .
\end{align*}

\item In the next case we want to illustrate the calculation of the relaxation times in the clustering algorithm.
Let $ J_k=\lbrace 1,3,5 \in \lbrace 1,\dots, 6 \rbrace: \tau_j \in [10^k, 10^{k+1}] \rbrace $ with $(\mu_1,\tau_1)=(7,4)$, $(\mu_3,\tau_3)=(1,4.5)$ and $(\mu_5,\tau_5)=(2,3.75)$, i.e. $\tau_1\approx \tau_3\approx \tau_5$. Then the stiffness values serve as a weighting of how important the corresponding relaxation value is. Regarding the forward operator 
\eqref{forward_op} we observe that, under these assumptions, the values of $\varepsilon(t) - \varepsilon_j^{\,i}(t)$, $t\in [0,T]$, for $j=1, 3, 5$ are quite similar, especially considering that other relaxation times lie in different decades. 
Thus, if one assumes that $\varepsilon_1^{i} \approx \varepsilon_3^{i} \approx \varepsilon_5^{i}$ applies, we see that the largest contribution among these three to the total stress is provided by the Maxwell element $(\mu_1,\tau_1)$, since $\mu_1 \gg\mu_5> \mu_3$. Therefore, we can assume that $\tau_1= 4$ provides a good approximation to the actual relaxation time. 
In order to include this weighting in more complicated examples, the clustering algorithm calculates $\tilde{\tau}_k$ as follows, instead of just calculating the mean value of all relaxation times,
\begin{align*}
\tilde{\mu}_k&= \sum_{j\in J_k} \mu_j= \mu_1+\mu_2+\mu_3=10,\\
\tilde{\tau}_k&= \sum_{j\in J_k} \frac{\mu_j}{\tilde{\mu}_k} \tau_j= \frac{7}{10}\cdot 4+ \frac{1}{10} \cdot 4.5+ \frac{2}{10} \cdot 3.75.= 4 . 
\end{align*}
\end{enumerate}
In the first case of our examples, this recalculation of the parameters by the clustering algorithm would not change anything in the generated stress curve. 
\end{example}

\section{Numerical results} \label{sec:NumericalResults}


In this section we show numerical results of our derived minimization and clustering algorithm. For the evaluation of our method we generate simulated data for known material parameters that serve as ground truth.
We choose a material described by one spring for basic elasticity and $n=3$ Maxwell elements and the following material parameters: 

\begin{table}[h]
\centering
	\caption{Selected material parameters to simulate data}
		\begin{tabular}{lcccc}\hline
			$j$&$0$ &$1$&$2$&$3$ \\ \hline
			 $\mu_j$ [MPa] &10&4&7&1 \\ \hline 
			 $\tau_j$ [s] &-&0.2&3.7&25 \\ \hline
		\end{tabular}
	\label{table:exact material parameters}
\end{table}

By solving the forward operator analytically (see Proposition \ref{P-exact-sigma}) we obtain the stress-time curve of the material, which we observe in the range $[0,100]$ seconds.
The parameters from Table \ref{table:exact material parameters} are to be reconstructed in all experiments. 
We discuss the difference between reconstructions from perturbed as well as exact data, different displacement rates and their effects. Finally we consider regularization, which is indispensable for perturbed data. 


\subsection{Exact data}


It will be shown that in case of exact data we can determine the material parameters very reliably. 
Table \ref{table:non_disturbed_data} shows the effects of the clustering algorithm.
While the minimization algorithm determines stiffnesses and relaxation times for the given maximum number of Maxwell elements (here $N=5$), the clustering algorithm can combine them accordingly and thus obtain the actual number $n=3$. Here we used the complete data set for $t \in [0,100]$ seconds and a displacement rate of $1$ mm/s, see Figure \ref{abb:strains}. 

\begin{table}[h]
\centering
	\caption{Reconstructed material parameters before and after clustering where the $\mu_j$ are given in MPa and the $\tau_j$ in seconds}
		\label{table:non_disturbed_data}
		\begin{tabular}{ccccccc}\hline
			$j=$&$0$ &$1$&$2$&$3$ & $4$ & $5$ \\ \hline
			reconstructed values \\ \hline
			 $\mu_j$ & 10.000 & 4.000 & 3.685 & 1.621 & 1.694 & 1.000 \\ \hline 
		      $\tau_j$ &-& 0.200 & 3.695 & 3.706 & 3.706 & 25.000 \\ \hline
		       after clustering \\ \hline
			 $\mu_j$ &10.000 & 4.000 & 7.000 & 1.000 \\ \hline
		 $\tau_j$ & - & 0.200 & 3.700 & 25.000 \\ \hline
		\end{tabular}
\end{table}

While after the calculation of the minimization algorithm there is only one Maxwell element with $\tau \in [0,1)$ and $\tau \in [10,100)$, we have three Maxwell elements with associated relaxation times in $[1,10)$. The clustering algorithm merges these elements and we can see that our reconstruction is accurate. The values in the table are rounded to three decimal places. 
After this rounding we no longer see any difference to the exact parameters after applying the clustering algorithm, see Table \ref{table:exact material parameters}. 
We illustrate what we have already mentioned in the second part of Example \ref{example: Cluster}. The relaxation time $\tau_2=3.695$ is closest to the current value of $3.7$ before clustering and we have $\mu_2 \gg \mu_4 > \mu_3$, which allows the clustering algorithm to use the weighting \eqref{eq:cluster calculation} leading to better results.


\subsection{Perturbed data} \label{sec: disturbed data}


We continue by investigating noisy data $\sigma^\delta$. These are generated by adding scaled standard normally distributed numbers to our discretized stress vector $\sigma$, such that
$\lVert \sigma- \sigma^\delta \rVert < \delta$ with noise level $\delta>0$. For better comparability we specify the relative noise level $\delta_{rel}$ by 
\begin{align*}
    \frac{\lVert \sigma- \sigma^\delta \rVert}{\lVert \sigma^\delta \rVert} < \delta_{rel}.
\end{align*} 
In our calculations we use the Euclidean norm and a relative noise level of $\delta_{rel} \approx 1 \%$.
Table \ref{table:noisy_data_10mm/s} shows different results of a reconstruction at a displacement rate of $10$\,mm/s using the clustering algorithm with $N=5$. To better assess the effect of the disturbance, we used several data sets for this purpose. That is, we have the same relative noise level $\delta_{rel}$, but different random numbers, which leads to slightly different data sets. We perform our reconstruction for each of these different data sets to identify possible weaknesses. 

	\begin{table}[h]
		\centering
			\caption{Approximated material parameters with similar noisy data, where the $\mu_j$ are given in MPa and the $\tau_j$ in seconds}
			\begin{tabular}{cccccc}\hline
				$j=$    && 	$0$	&$1$	&$2$	&$3$ 		\\ \hline
				exact values	&$\mu_j$		&10			&4						&7		&1 			\\   
								&$\tau_j$	&-			&0.2	&3.7	&25 		\\ \hline
				experiment 1&$\mu_j$		&9.9951		&38.558					&7.2098	&0.88371 		\\  
								&$\tau_j$	&-			& 0.019903	&3.7734	&27.795 		\\ \hline
				experiment 2	 &$\mu_j$		&10.005	&2.475					&6.9957	&1.0157 		\\  
							 	&$\tau_j$	&-			&0.29222	&3.6908	&24.2965 		\\ \hline
				experiment 3 &$\mu_j$		&9.99348	&44.351					&7.2084	&0.96407 		\\  
								&$\tau_j$	&-			&0.01344	&3.6444	&26.633 		\\ \hline
				experiment 4	&$\mu_j$		&10.003		&7.3458					&7.0398	&0.99887 		\\   
								&$\tau_j$	&-			&0.0993 & 3.6626	&24.0342 	\\ \hline
			\end{tabular}
					\label{table:noisy_data_10mm/s}
	\end{table}
	
We see that the reconstruction results differ, as expected, significantly from the exact values. However, it is obvious that the reconstructions of $(\mu_1,\tau_1)$ show a tremendous error, while the other parameters are computed rather stably. This confirms considerations of the authors in \cite{diebels2018identifying} where the authors have proven that the reconstruction of small relaxation times always is severely ill-conditioned and the condition deteriorates as $\tau\to 0$. Since stiffness and relaxation time are always to be considered in pairs, we see deviations in $\tau_1$ also in the corresponding stiffness $\mu_1$. As each Maxwell element has to provide a certain part to the total stress, too small values in one parameter are compensated by higher values in the other parameter to approach the total stress.

\begin{figure}[h]
	\centering
		\def\svgwidth{\textwidth}
		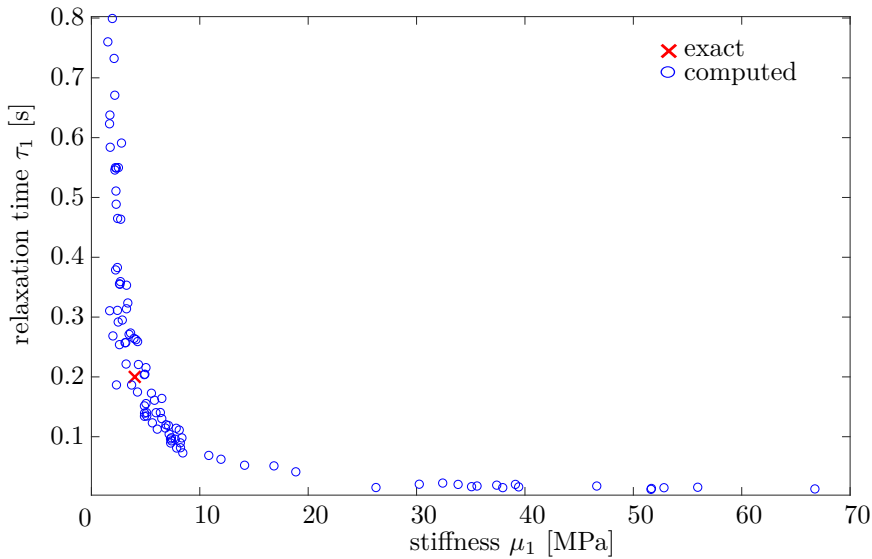
	\caption{Spread of the material parameters $(\mu_1,\tau_1)$ for 100 runs with different noise seeds, but same noise level $\delta_{rel}$ for $\eta$ = 10 mm/s}
	\label{fig:100runs_first_element_spread}
\end{figure}

We repeat this experiment with 100 different perturbed data sets with the same noise level and plot the values of $(\mu_1,\tau_1)$ in Figure \ref{fig:100runs_first_element_spread}.
We see that there is a very large variance in values spread over 100 different experiments. The values are distributed up to results such as $(\mu_1,\tau_1)=(66.75, 0.0127)$ or $(\mu_1,\tau_1)=(1.934, 0.7992)$. Of course, there are also results that come close to the exact values. The median of all values is $(0.1897,4.0135)$, which is not far from the exact values. 
Figure \ref{fig:100runs_first_element_spread} also shows that stiffness and relaxation time are inversely proportional to each other. If the relaxation time is estimated too small, we obtain much too large values for the stiffness. Similarly, if the relaxation time is too large, the stiffness is reconstructed too small. 

We summarize that the reconstruction from noisy data is, as expected, much worse in comparison to the results obtained for exact data. In particular, the Maxwell element with the smallest relaxation time (here: $(\mu_1,\tau_1)$) is most affected. This confirms stability investigations obtained in \cite{diebels2018identifying}.


\subsection{Analysis of different displacement rates}\label{sec:different strain rates}


So far we focused on experiments with the fastest displacement rate of $10$\,mm/s. In this subsection we discuss slower displacement rates. 
	\begin{table}[h]
	\centering
		\caption{Clustered material parameters with noisy data ($\eta = 1$ mm/s) where the $\mu_j$ are given in MPa and the $\tau_j$ in seconds}
		\begin{tabular}{c|c|c|c|c|c} \hline
			$j=$&&$0$ &$1$&$2$&$3$ \\  \hline
		exact values	&$\mu_j$		&10			&4		&7		&1 			\\ 
						&$\tau_j$	&-			&0.2	&3.7	&25 		\\ \hline 
			
		experiment 1	&$\mu_j$		&9.9945		&27.367	&5.945	&0.91663 		\\   
						&$\tau_j$	&-			&0.088687	&4.2318	&27.263 		\\ \hline
			
		experiment 2&$\mu_j$		&10.007		& - &7.5694	&1.0434 		\\ 
						&$\tau_j$	&-			&- &3.4762	&23.689			\\  \hline
			
		experiment 3&$\mu_j$		&9.0111		&2.6942		&7.2928	&1.0402	 	\\ 
						&$\tau_j$	&-			&0.82841		&6.3049	&3000.5	 	\\ \hline
			
		experiment 4 	&$\mu_j$		&10.006		&2.094		&6.0643	&1.0303	 	\\ 
						&$\tau_j$	&-			&0.885029		&3.9963	&23.423	 	\\ \hline
		\end{tabular}
			\label{table:four_res_strain1}
\end{table}

We compare Table \ref{table:noisy_data_10mm/s} with Table \ref{table:four_res_strain1}, in which we used a displacement rate of 10\,mm/s and 1\,mm/s, respectively. Again, we use $N=5$ for the minimization and a relative noise level of about 1\%, but four different noise seeds. It is obvious that the results are much worse, specifically the smallest relaxation time and the corresponding stiffness are approximated poorly. Again, this has to be expected since the reconstruction of small relaxation times is severely ill-conditioned. In a second run, the algorithm was not even able to find Maxwell elements with relaxation times in three different decades.

\begin{figure}[h]
\centering
\def\svgwidth{0.8\textwidth}
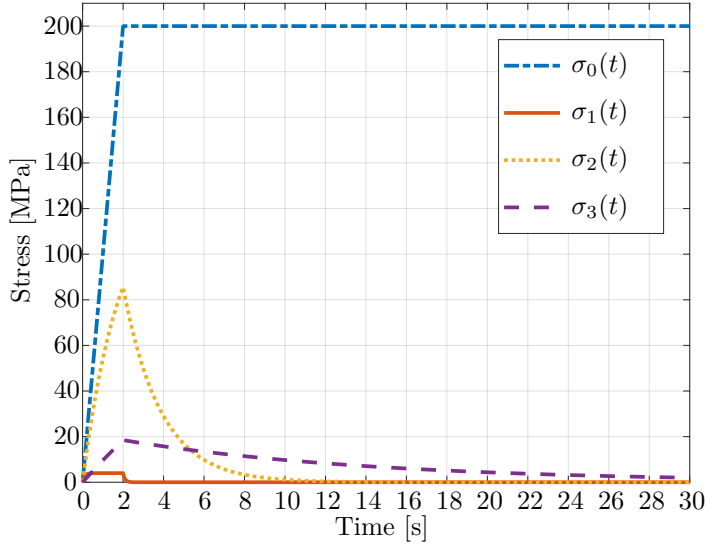
\caption{Individual stress components for a displacement rate of 10 mm/s}
\label{abb:10mmps}
\end{figure}

\begin{figure}[h]
\centering
\def\svgwidth{0.8\textwidth}
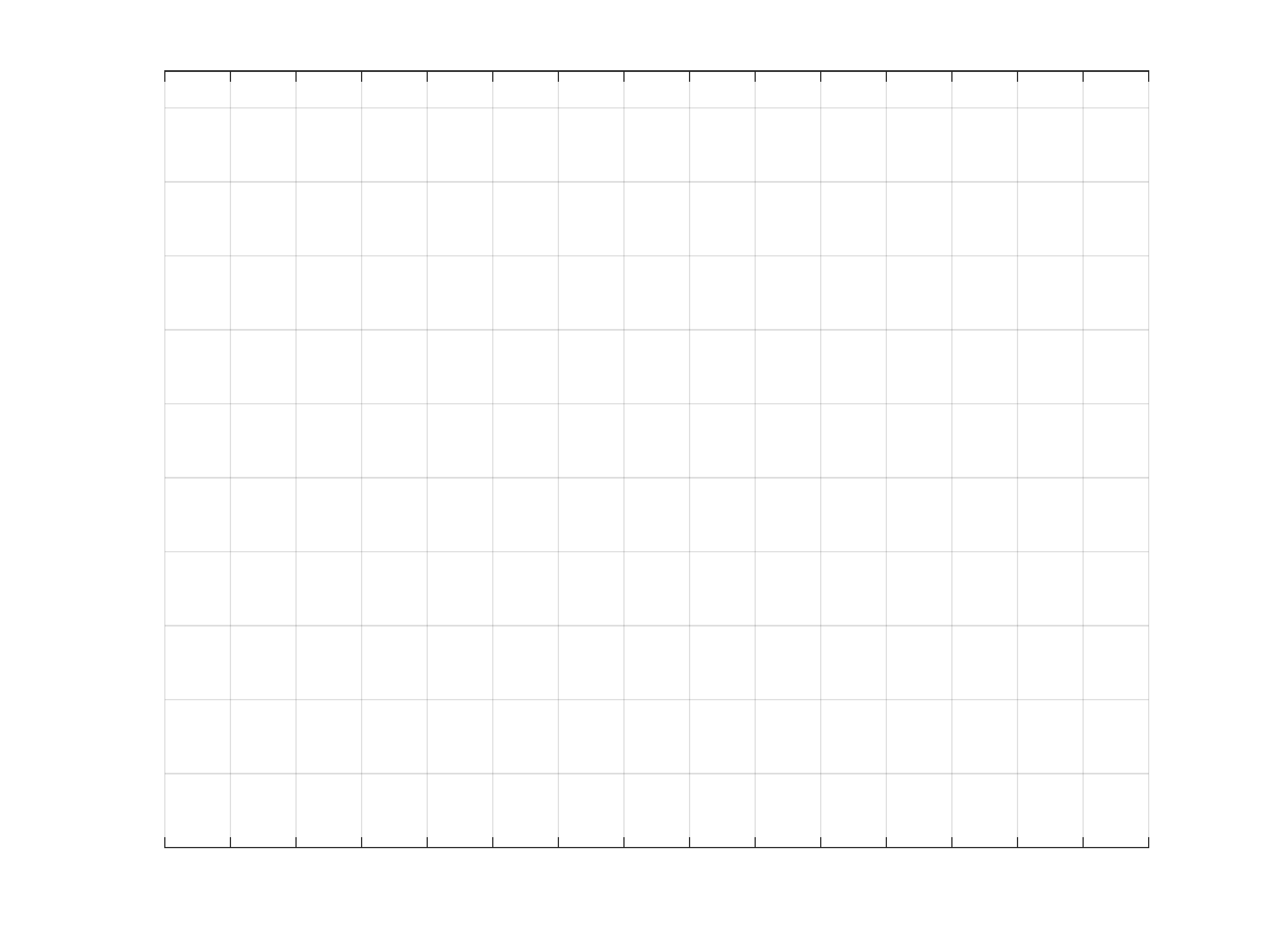
\caption{Individual stress components for a displacement rate of 1 mm/s }
\label{abb:1mmps}
\end{figure}

For a better understanding of why the results are worse compared to the results for the faster displacement rate of 10\,mm/s, we look at the different stress-time curves that a displacement rate of 1\,mm/s (Figure \ref{abb:10mmps}) and 10\,mm/s (Figure \ref{abb:10mmps}) generate. For each of the curves, we show the contribution of each element separately. 
That means $\sigma_0$ is the stress generated by the single spring while $\sigma_{1,2,3}$ is the stress of the three Maxwell elements. The sum of these individual stresses yields the total stress, cf. \eqref{eq: sigmaTotal}.
The maximum strain is 20 $\%$, which is attained at 20 seconds and 2 seconds for a displacement rate of 1 mm/s and 10 mm/s, respectively. This makes a big difference for the individual stresses of the Maxwell elements. In both cases the maximum stress of the single spring is achieved at 200 MPa, but at different time instances. 
However, the maximum stress of the Maxwell elements at a slower displacement rate of 1 mm/s is much lower than for a faster displacement rate. In both cases, the maximum is attained at $t=\bar{\varepsilon}/\eta$. In Table \ref{table: maxStress ME different strain rates}, we see the different values listed. 
At a slower displacement rate, the stress values of the Maxwell elements are non-zero over a longer time period, which can be seen particularly well for the first Maxwell element. But since the values are so much smaller than at a higher displacement rate, these small values are covered by noise much more easily. 
Therefore, it is advisable to use higher displacement rates. Based on this, we revert to a displacement rate of 10 mm/s in upcoming examples. 

\begin{table}	[h]
	\centering
		\caption{Maximum values of individual stress components [MPa] with slow and fast deformation rates attained at
		$t=\bar{\varepsilon}/\eta$}
	\begin{tabular}{c|c|c|c|c} \hline	&	&$j=1$	&$j=2$	&$j=3$ 			\\ \hline
	$\eta = 1$ 	&	$\sigma_j(\bar{\varepsilon}/\eta)$				&0.4	&12.94	&9.9 		\\
				 \hline
		
	$\eta = 10$	&	$\sigma_j(\bar{\varepsilon}/\eta)$				&4		&85.57	&18.48	 	\\ \hline 
	\end{tabular}
		\label{table: maxStress ME different strain rates}
\end{table}


\subsection{Regularization approaches} \label{sec: regularization}


As a result of our investigations in Section \ref{sec: disturbed data} we realized that small errors in the data can lead to large errors in the solution. This is a typical behavior of ill-posed inverse problems, which must be compensated by means of regularization techniques.
Therefore, instead of continuing to use as cost function the residual \eqref{eqn: Residuum} only, we will add a regularization term and minimize 
\begin{align} \label{eqn: Tikhonov}
\lVert F\left(\mu, \mu_1, \dots, \mu_N, \tau_1, \dots \tau_N\right) - \sigma^\delta \rVert_2^2 + \lambda \lVert \left(\mu, \mu_1, \dots, \mu_N, \tau_1, \dots \tau_N\right) \rVert_2^2,
\end{align} 
where $\lambda>0$ acts as regularization parameter and balances the influence of the data fitting term and the penalty term to the minimizer. This is the \emph{Tikhonov-Phillips regularization}, a standard approach in the field of inverse problems (see \cite{engl1989convergence,kaltenbacher2008iterative, neubauer1992tikhonov, scherzer1993use}). In Table \ref{tab:res_with_and_without_reg}, we see the comparison of two reconstructions with and without the regularization term. Again we used $N=5$ and the clustering algorithm for these results. We see a significant improvement on the first Maxwell element, which was previously most affected by the disturbed data. As expected, too large values are suppressed in the stiffness, but at the same time $\tau_1$ attains larger values. However, we also see a deterioration in the values of the third Maxwell element. Since large values are penalized by the regularization term, this mainly affects the relaxation time $\tau_3=25$ s. The values are lower and we get a higher stiffness $\mu_3=1.1312$ MPa and a smaller relaxation time $\tau_3=22.2532$ s. 

\begin{table}[h] 
	\centering
	\caption{Clustered material parameters with and without regularization where the $\mu_j$ are given in MPa and the $\tau_j$ in seconds}
	\label{tab:res_with_and_without_reg}
		\begin{tabular}{cccccc} \hline
						$j=$	&			&$0$ 		&$1$	&$2$	&$3$ 		\\ \hline  
		exact values		&$\mu_j$		&10			&4		&7		&1 			\\ 
							&$\tau_j$	&-			&0.2	&3.7	&25 		\\ \hline  
			
without						&$\mu_j$		&10.0003	&6.9146					&6.9963	&0.9909 		\\  
regularization				&$\tau_j$	&-			&0.1096	&3.7056	&25.3752 		\\ \hline 
			
		with				&$\mu_j$		&10.0065		&4.214 &6.9556	&1.1312	 	\\ 
		regularization		&$\tau_j$	&-			&0.1607	&3.5549	&22.2532	 	\\ \hline  
		\end{tabular}
\end{table}

To avoid this, we test another penalty term. We minimize
\begin{align} \label{eqn: Strafterm mu1}
\lVert F\left(\mu, \mu_1, \dots, \mu_N, \tau_1, \dots \tau_N\right) - \sigma^\delta \rVert_2^2 + \lambda \mu_1 ^2.
\end{align} 
Since $\mu_1$ is just a non-negative number in $\mathbb{R}$, we can omit the norm. 
The idea is to have the same positive effects as the classical Tikhonov-Phillips regularization, but without the negative influence on the largest relaxation time and thus the associated stiffness. 
In Figures \ref{abb:comp tau1 regularization} and \ref{abb:comp tau3 regularization} we compare all three methods. 
As in Section \ref{sec: disturbed data} we evaluate the different approaches in 100 experiments with the same noise level but different noise seeds. The clustered results are displayed in Figures \ref{abb:comp tau1 regularization} and \ref{abb:comp tau3 regularization}. In each box the central marker indicates the median, the lower and upper end of the box mark the 25th and 75th percentile, respectively. The outer boundaries mark the highest and smallest values. Outliers are marked by the red '+' symbols. 

\begin{figure}[h]
\centering
\def\svgwidth{0.7\textwidth}
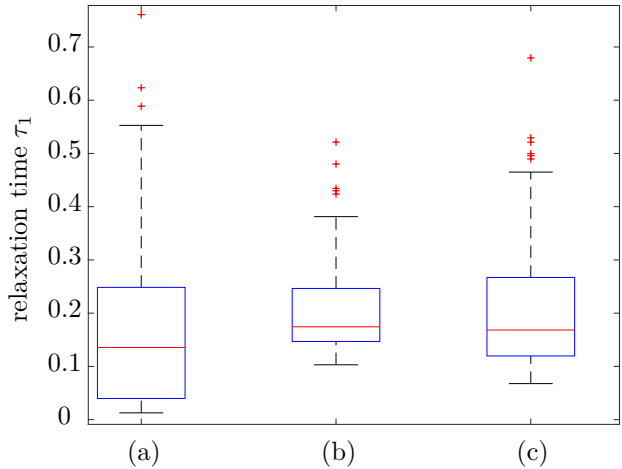
\caption{Smallest relaxation time determined from 100 runs with (a) no regularization, (b) classical Tikhonov-Phillips regularization and (c) adjusted regularization term}
\label{abb:comp tau1 regularization}
\end{figure}

\begin{figure}[h]
\centering
\def\svgwidth{0.7\textwidth}
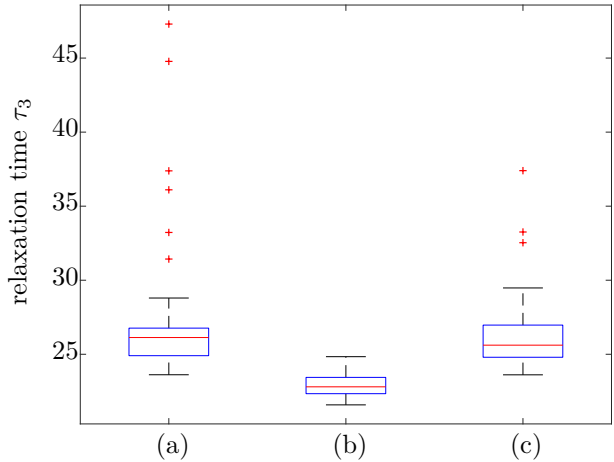
\caption{Largest relaxation time for 100 runs with (a) no regularization, (b) classical Tikhonov-Phillips regularization and (c) adjusted regularization term}
\label{abb:comp tau3 regularization}
\end{figure}


We see that we still manage to achieve a reasonable reconstruction of $\tau_1$ with the adjusted regularization term (see Figure \ref{abb:comp tau1 regularization} (c)). 
While $\tau_3$ was permanently reconstructed too small in the classical Tikhonov-Phillips regularization, we cannot observe such attenuation in (c), since we no longer penalize large values except $\mu_1$ (see Figure \ref{abb:comp tau3 regularization}). 


\subsection{Shortened data sets}


the last set of numerical experiments examines the effects of using a smaller number of data.
This is interesting in view of shortening the duration of the experiment or sampling the data at a lower frequency. 
It is particularly relevant for the reconstruction of Maxwell elements with high relaxation times. Since we do not know these, we cannot estimate how long it takes for the stress contribution of this element to decay to zero. In other words: It is unclear at which time a material finds its equilibrium. This is why in the following we will test reconstructions from shortened data sets in order to be able to make an assertion about the accuracy of these reconstructions.

We again work with noisy data ($\delta_{rel} \approx 1\%$), $N=5$, for the minimization and use the adapted penalty term \eqref{eqn: Strafterm mu1}.
Up to now, we have used a time interval of $[0,100]$ seconds with 1000 data points in our experiments, i.e. one data point any $0.1$ seconds.
We will now repeatedly shorten this data set by 5 seconds until we get a total time of 25 seconds. 
The results of the corresponding reconstructions are shown in Figure \ref{fig:shortenedData}. We focus on the basic elasticity $\mu$ and stiffness $\mu_3$ of the Maxwell element with slowest relaxation time and furthermore on the smallest and largest relaxation times $\tau_1$ and $\tau_3$. The behavior of the central Maxwell element is comparable to the third one, which is why we omit a detailed discussion.

\begin{figure}
     \centering
     \subfloat{\def\svgwidth{0.5\textwidth}
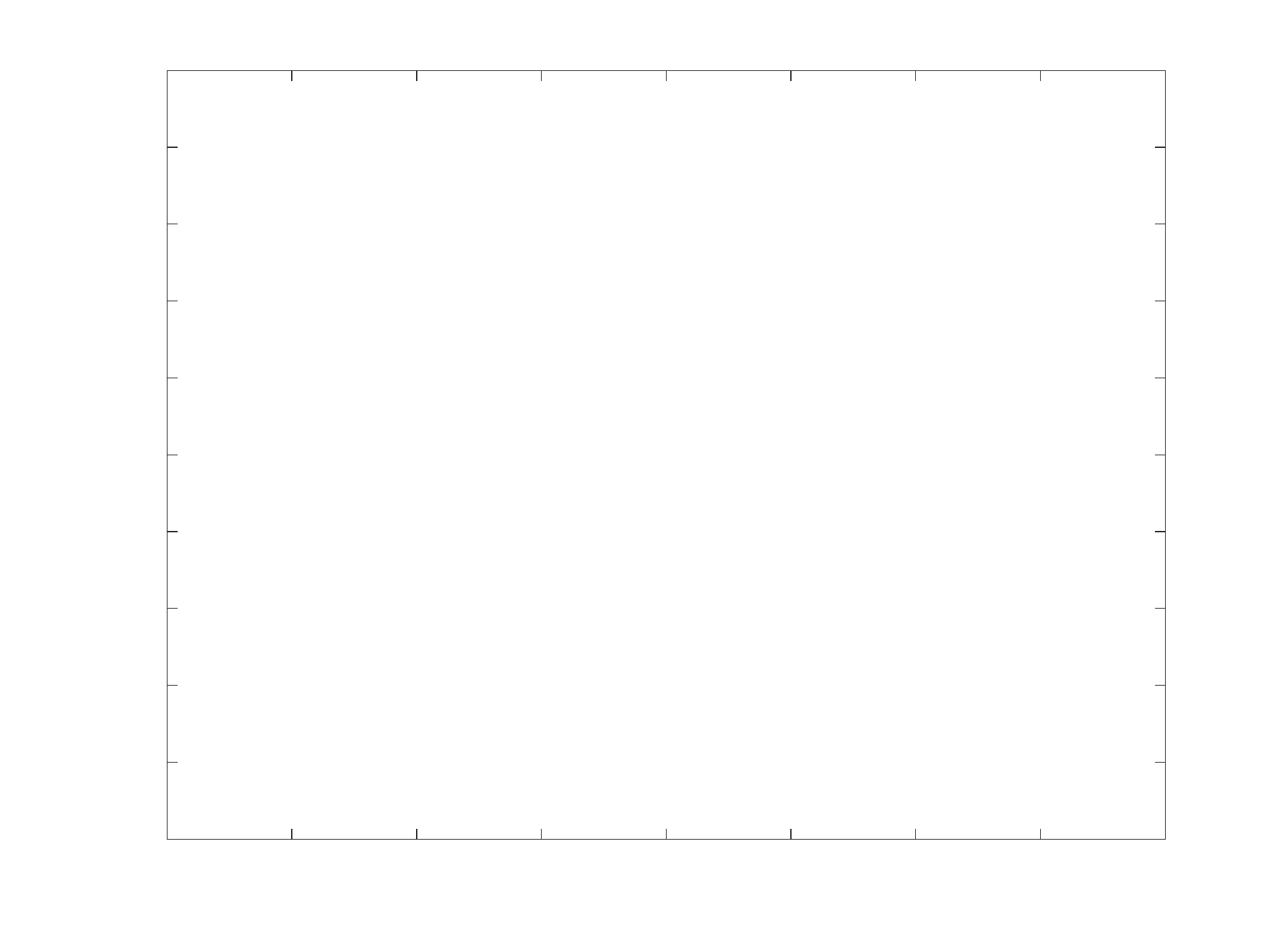}
     \subfloat{\def\svgwidth{0.5\textwidth}
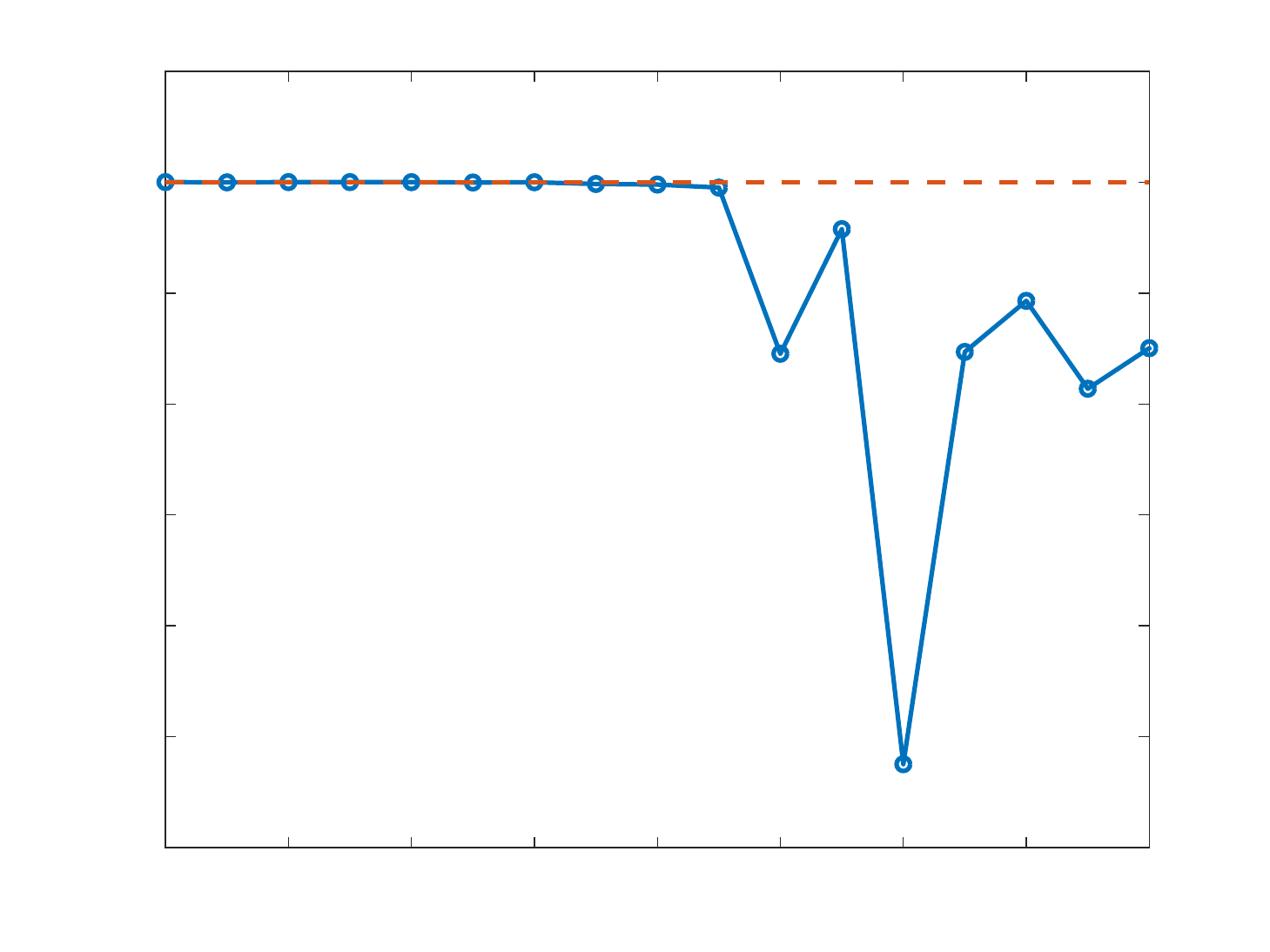}\\
     \subfloat{\def\svgwidth{0.5\textwidth}
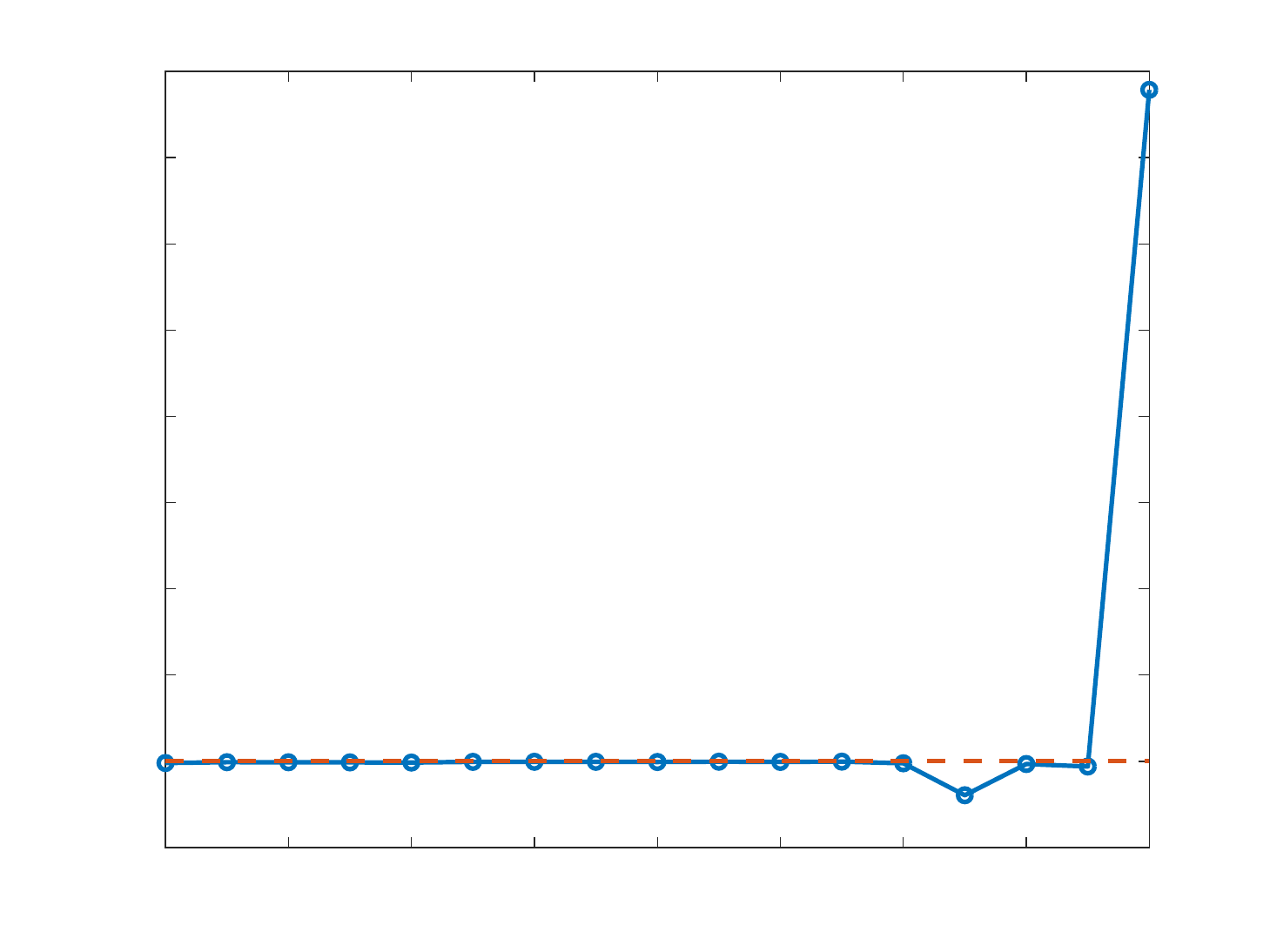}
     \subfloat{\def\svgwidth{0.5\textwidth}
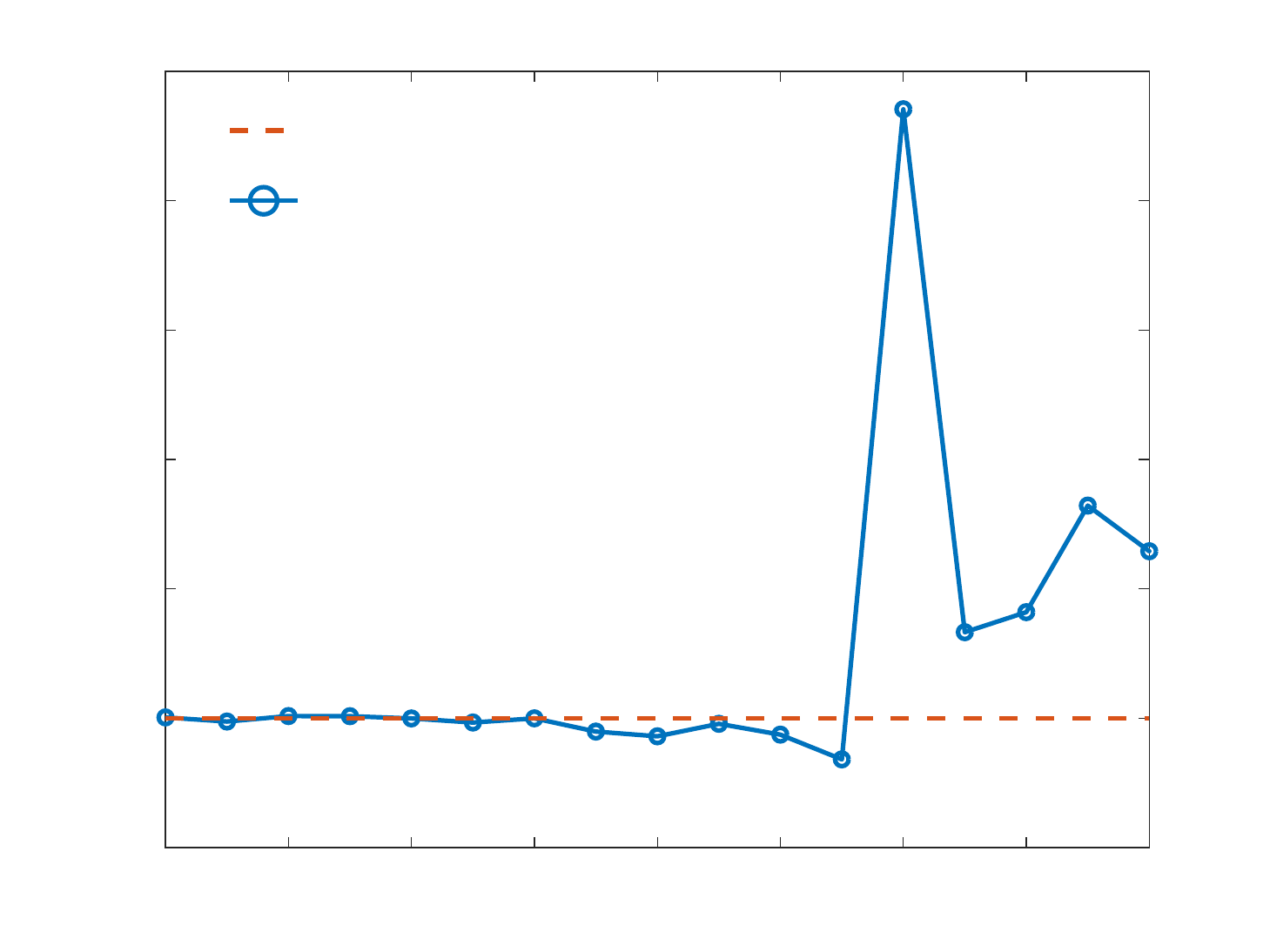}\\
     \subfloat{\def\svgwidth{0.5\textwidth}
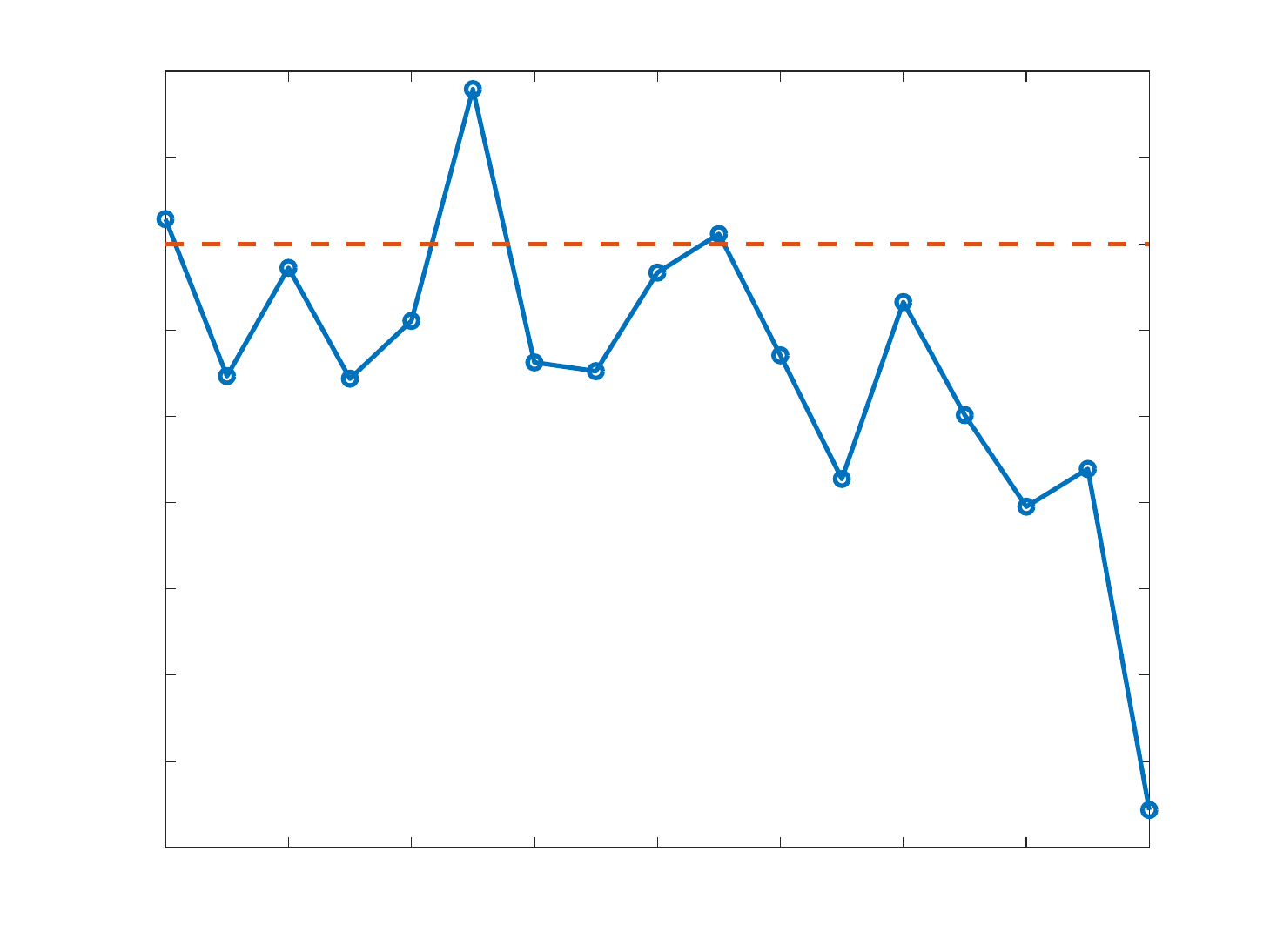}
     \subfloat{\def\svgwidth{0.5\textwidth}
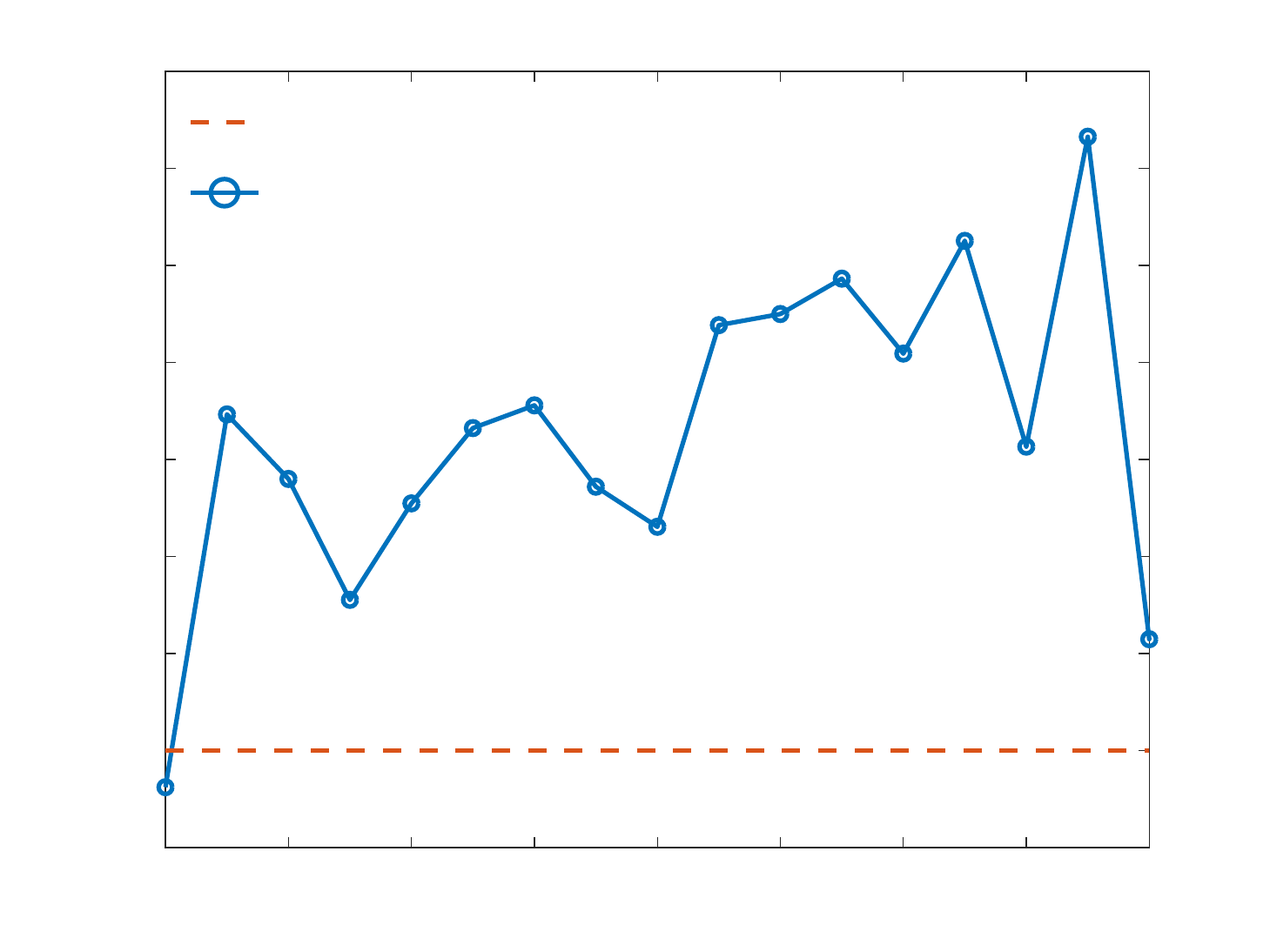}\\
     \subfloat{\def\svgwidth{0.5\textwidth}
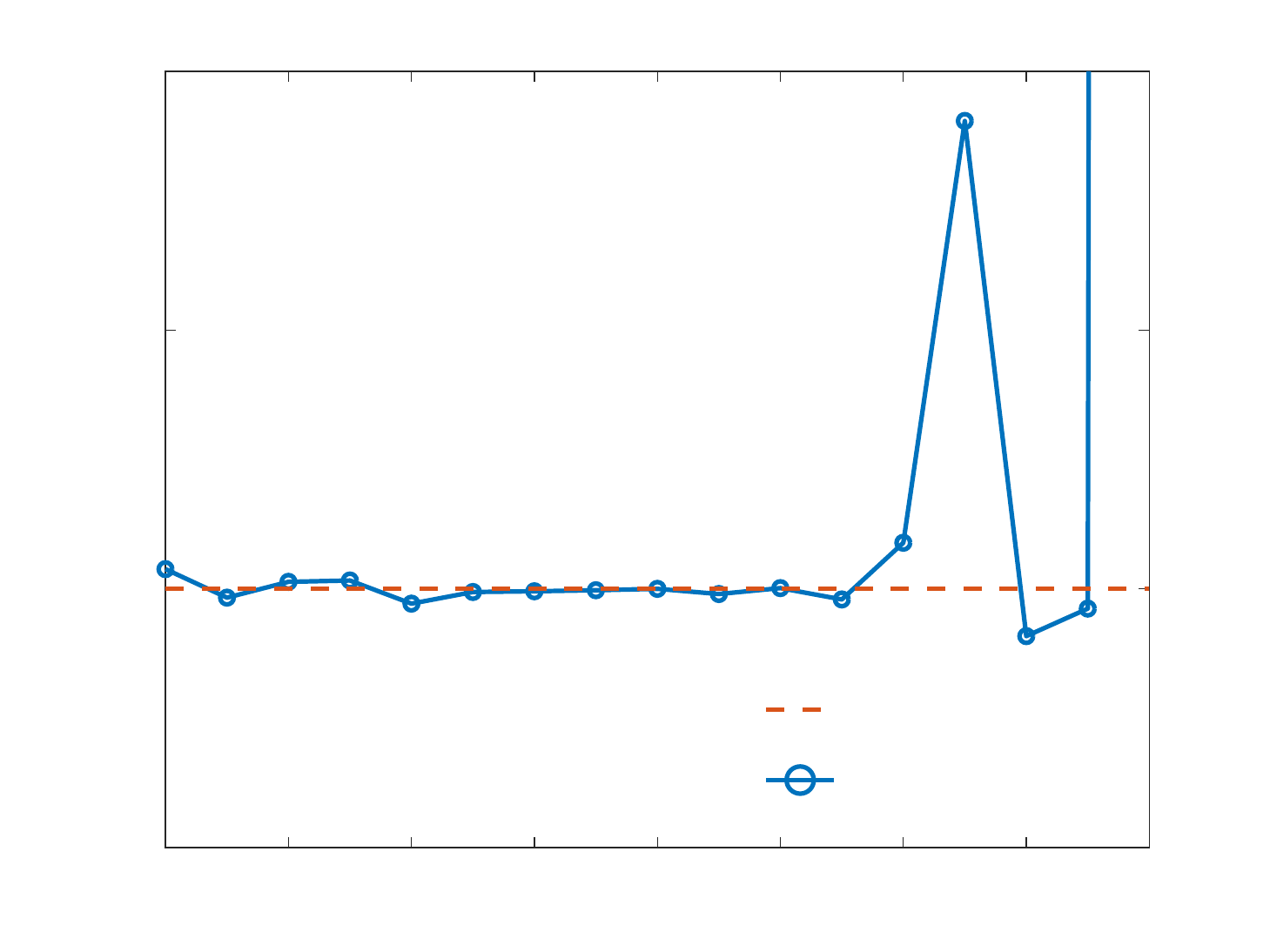}
     \subfloat{\def\svgwidth{0.5\textwidth}
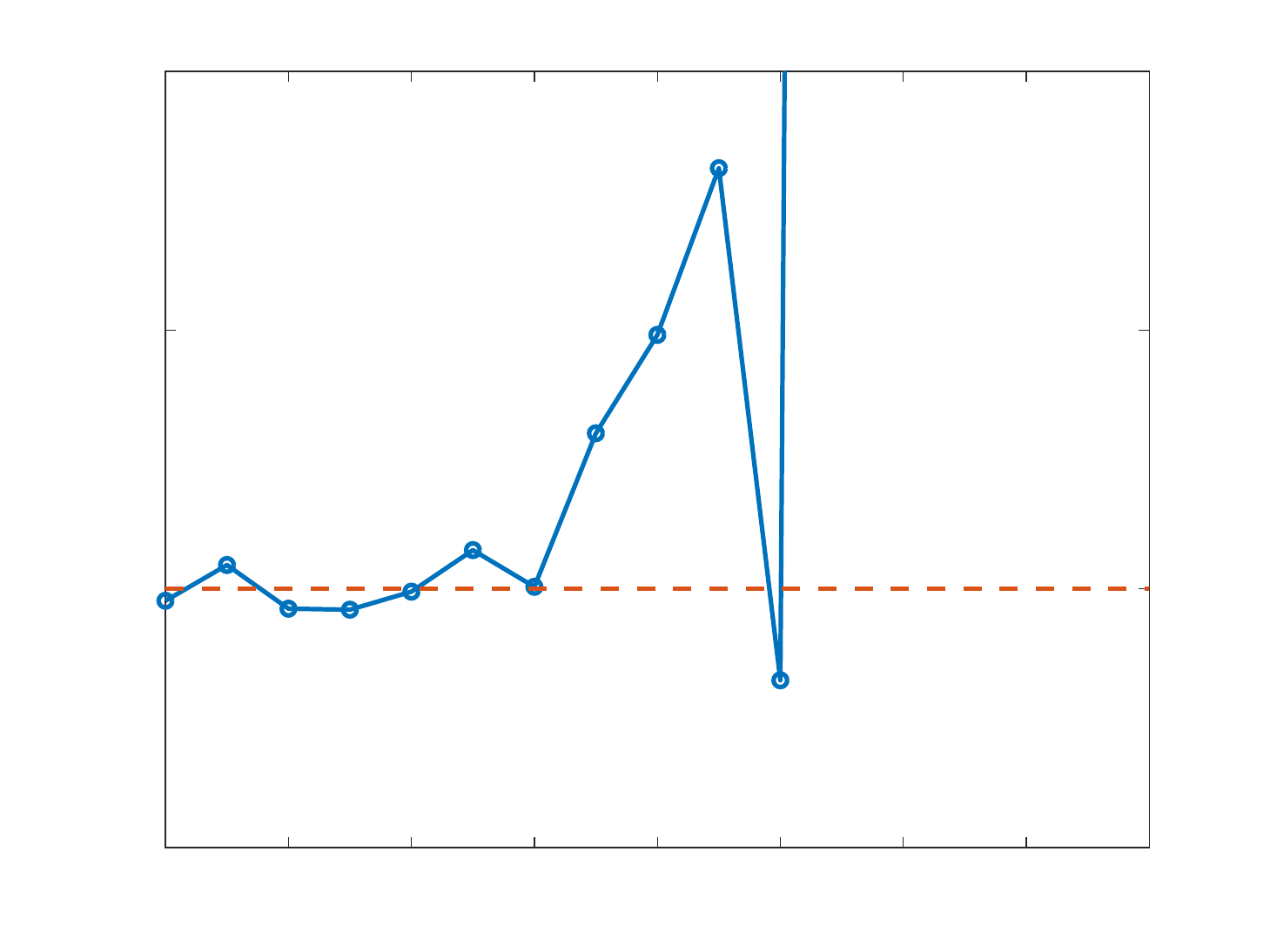}
\caption{Effect of reducing the duration $[0,T]$ of the experiment on the material parameters $\mu, \mu_3, \tau_1$ and $\tau_3$ for a displacement rate of 10 mm/s (left) and 1 mm/s (right), respectively.}
\label{fig:shortenedData}
\end{figure}

The basic elasticity $\mu$ of the single spring can perfectly be identified after only 25 seconds for a displacement rate of 10 mm/s. For the slower rate of 1 mm/s, 55 seconds of the experiment are necessary for a reliable identification. The minimal time to conclusively identify the stiffness $\mu_3$ of the spring in the third Maxwell element takes 40 seconds for a rate of 10 mm/s and 50 seconds for a rate of 1 mm/s, respectively. 

The identification of the relaxation times is more difficult than of the stiffnesses. The relaxation times $\tau_3$ of the third Maxwell element can be precisely determined after 45 seconds for a rate of 10 mm/s and after 70 seconds for the slower rate of 1 mm/s, respectively. After that the result gets worse with every second the data set is shortened. There are not sufficient data points left to perform a reliable reconstruction. The variation in the values for the fastest relaxation time $\tau_1$ is much larger than for $\tau_3$ and also very sensitive to noise. Experimental observations show in general that, if the displacement rate is too high, the third Maxwell element with the slowest relaxation time behaves like an equilibrium spring since there is not enough time to relax during loading. The stiffness of the third element superimposes with the equilibrium spring. However, during relaxation the equilibrium spring and the third Maxwell element can be easily distinguished from each other. Hence, the faster loading rate of 10 mm/s can predict $\mu$ and $\mu_3$ even with 30 s of the data whereas, for the 1 mm/s loading rate, at least 50 s of data is required for accurate prediction. On the other hand the prediction of $\tau_1$ is easier with the 10 mm/s data because, if the displacement rate is too low, the first Maxwell element with the fastest relaxation time starts to relax during the loading itself. This is shown for a rate of 1 mm/s in Figure \ref{fig:shortenedData}. A satisfying identification is not possible during the entire experimental duration of 100 seconds. However, for the faster rate of 10 mm/s $\tau_1$ it can already be determined after 50 seconds.

As a summary of the observations above and for a better understanding, Figure \ref{abb:stress_repeat} again shows the stress-time curves for the three different displacement rates, but now including the time in the experiments to identify the individual material parameters. 
The parameters are not of the same order for the different rates. For instance, for a displacement rate of 1mm/s the relaxation time $\tau_3$ can be reconstructed first, while at 10mm/s the stiffness $\mu$ is most stable to data shortening. Thus, no general statement can be made about a correlation of the material parameters and the duration of the experiment without considering conditions such as the displacement rate. 

\begin{figure}[h]
\centering
\def\svgwidth{.85\linewidth}
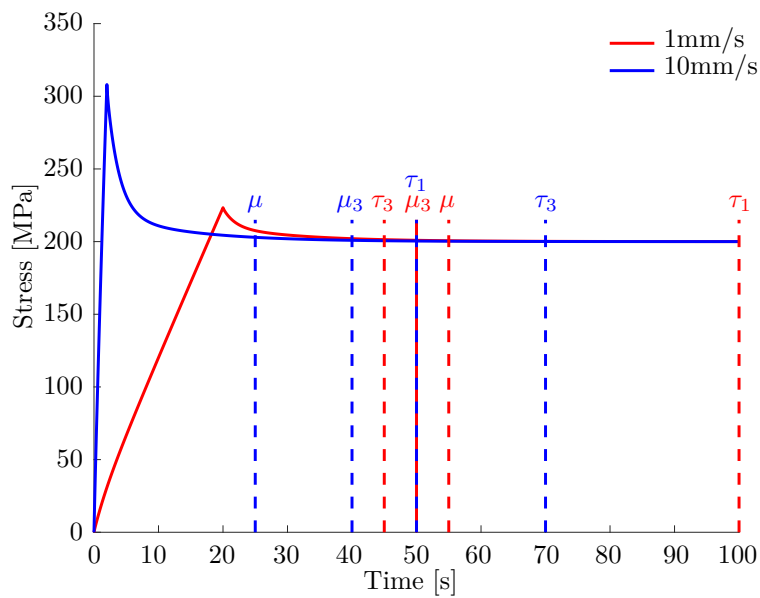
\caption{Synthetic stress-time data produced by a spring combined with a three Maxwell element model including the duration of the experiment to conclusively identify the stiffnesses and relaxation times for displacement rates of 1mm/s (red) and 10mm/s (blue) }
\label{abb:stress_repeat}
\end{figure}






\section{Conclusion and future work}


In this paper, we presented an algorithm for the reconstruction of material parameters in a Maxwell-based rheological multi-parameter model for viscoelastic materials. After a short introduction to the material model, we defined the forward operator and proposed an analytical way to solve the evolution equation with respect to relaxation times. This allowed us to determine an analytical solution of the forward operator with regard to the number of Maxwell elements.

Since the number of Maxwell elements is not known, we suggested a method to overcome the solution-dependence of the forward operator by using a clustering algorithm.
For the parameter identification, we considered different approaches based on minimizing a least squares functional. In experiments we have seen that the simple approach to minimize the residual is not sufficient for perturbed data and especially the material parameters of the Maxwell element with shortest relaxation time is very susceptible to noise.
By analyzing the stress-time curve decomposed into the single spring and Maxwell element stresses, we were able to give a clear recommendation towards higher displacement rates. 
This analysis also contributed to the understanding of why the small relaxation times with their corresponding stiffness are that sensitive to noise. 

Subsequently, we considered various regularization methods such as the classical Tikhonov-Phillips approach. We conclude that a specific regularization adapted to the smallest relaxation time is preferable.

Investigations on shortened data sets demonstrate that our approach for material parameter identification shows great advantages for performing real experiments since it yields suggestions on the minimum duration of the experiments. This is very important, because typical relaxation experiments can last days or weeks until the basic elasticity can be determined using conventional evaluation methods. Whether the duration of an experiment is sufficient for the identification of further parameters is not obvious by conventional approaches. In that sense our new approach might be of significant interest for the identification of parameters in materials science.

Future research is focused on validating the presented algorithms on real data sets. Furthermore, it might be interesting to apply other numerical regularization methods, such as the Landweber iteration or Newton type methods.


\bibliographystyle{siam}
\bibliography{ParameterOptim}   

\end{document}

%% file: StrainExplained.pdf_tex
\begingroup%
  \makeatletter%
  \providecommand\color[2][]{%
    \errmessage{(Inkscape) Color is used for the text in Inkscape, but the package 'color.sty' is not loaded}%
    \renewcommand\color[2][]{}%
  }%
  \providecommand\transparent[1]{%
    \errmessage{(Inkscape) Transparency is used (non-zero) for the text in Inkscape, but the package 'transparent.sty' is not loaded}%
    \renewcommand\transparent[1]{}%
  }%
  \providecommand\rotatebox[2]{#2}%
  \newcommand*\fsize{\dimexpr\f@size pt\relax}%
  \newcommand*\lineheight[1]{\fontsize{\fsize}{#1\fsize}\selectfont}%
  \ifx\svgwidth\undefined%
    \setlength{\unitlength}{420bp}%
    \ifx\svgscale\undefined%
      \relax%
    \else%
      \setlength{\unitlength}{\unitlength * \real{\svgscale}}%
    \fi%
  \else%
    \setlength{\unitlength}{\svgwidth}%
  \fi%
  \global\let\svgwidth\undefined%
  \global\let\svgscale\undefined%
  \makeatother%
  \begin{picture}(1,0.75)%
    \lineheight{1}%
    \setlength\tabcolsep{0pt}%
    \put(0,0){\includegraphics[width=\unitlength,page=1]{StrainExplained.pdf}}%
    \put(0.09900378,0.02990773){\color[rgb]{0,0,0}\makebox(0,0)[lt]{\lineheight{1.25}\smash{\begin{tabular}[t]{l}$t=0$\end{tabular}}}}%
    \put(0.86999634,0.02990774){\color[rgb]{0,0,0}\makebox(0,0)[lt]{\lineheight{1.25}\smash{\begin{tabular}[t]{l}$t=T$\end{tabular}}}}%
    \put(0.48303607,0.03125){\makebox(0,0)[lt]{\lineheight{1.25}\smash{\begin{tabular}[t]{l}Time [s]\end{tabular}}}}%
    \put(0,0){\includegraphics[width=\unitlength,page=2]{StrainExplained.pdf}}%
    \put(0.09165754,0.27146539){\rotatebox{90}{\makebox(0,0)[lt]{\lineheight{1.25}\smash{\begin{tabular}[t]{l}Strain [\%]\end{tabular}}}}}%
    \put(0,0){\includegraphics[width=\unitlength,page=3]{StrainExplained.pdf}}%
    \put(0.08469387,0.66122449){\color[rgb]{0,0,0}\makebox(0,0)[lt]{\lineheight{1.25}\smash{\begin{tabular}[t]{l}$\bar{\varepsilon}$\end{tabular}}}}%
    \put(0.23349887,0.03058748){\color[rgb]{0,0,0}\makebox(0,0)[lt]{\lineheight{1.25}\smash{\begin{tabular}[t]{l}$t=\frac{\bar{\varepsilon}}{\eta}$\end{tabular}}}}%
  \end{picture}%
\endgroup%

%% file: RheologicalModel.pdf_tex
\begingroup%
  \makeatletter%
  \providecommand\color[2][]{%
    \errmessage{(Inkscape) Color is used for the text in Inkscape, but the package 'color.sty' is not loaded}%
    \renewcommand\color[2][]{}%
  }%
  \providecommand\transparent[1]{%
    \errmessage{(Inkscape) Transparency is used (non-zero) for the text in Inkscape, but the package 'transparent.sty' is not loaded}%
    \renewcommand\transparent[1]{}%
  }%
  \providecommand\rotatebox[2]{#2}%
  \newcommand*\fsize{\dimexpr\f@size pt\relax}%
  \newcommand*\lineheight[1]{\fontsize{\fsize}{#1\fsize}\selectfont}%
  \ifx\svgwidth\undefined%
    \setlength{\unitlength}{422.87974548bp}%
    \ifx\svgscale\undefined%
      \relax%
    \else%
      \setlength{\unitlength}{\unitlength * \real{\svgscale}}%
    \fi%
  \else%
    \setlength{\unitlength}{\svgwidth}%
  \fi%
  \global\let\svgwidth\undefined%
  \global\let\svgscale\undefined%
  \makeatother%
  \begin{picture}(1,0.50101099)%
    \lineheight{1}%
    \setlength\tabcolsep{0pt}%
    \put(0,0){\includegraphics[width=\unitlength,page=1]{RheologicalModel.pdf}}%
    \put(-0.00178972,0.22411127){\color[rgb]{0,0,0}\makebox(0,0)[lt]{\lineheight{1.25}\smash{\begin{tabular}[t]{l}$\mu$\end{tabular}}}}%
    \put(0.37330229,0.48346024){\color[rgb]{0,0,0}\makebox(0,0)[lt]{\lineheight{1.25}\smash{\begin{tabular}[t]{l}j=1\end{tabular}}}}%
    \put(0.60333587,0.48346024){\color[rgb]{0,0,0}\makebox(0,0)[lt]{\lineheight{1.25}\smash{\begin{tabular}[t]{l}j=2\end{tabular}}}}%
    \put(0.82165571,0.48229897){\color[rgb]{0,0,0}\makebox(0,0)[lt]{\lineheight{1.25}\smash{\begin{tabular}[t]{l}j=3\end{tabular}}}}%
    \put(0.26181983,0.11544551){\color[rgb]{0,0,0}\makebox(0,0)[lt]{\lineheight{1.25}\smash{\begin{tabular}[t]{l}$\mu_1$\end{tabular}}}}%
    \put(0.48482461,0.11544551){\color[rgb]{0,0,0}\makebox(0,0)[lt]{\lineheight{1.25}\smash{\begin{tabular}[t]{l}$\mu_2$\end{tabular}}}}%
    \put(0.71824098,0.11544551){\color[rgb]{0,0,0}\makebox(0,0)[lt]{\lineheight{1.25}\smash{\begin{tabular}[t]{l}$\mu_3$\end{tabular}}}}%
    \put(0.42439845,0.31053644){\color[rgb]{0,0,0}\makebox(0,0)[lt]{\lineheight{1.25}\smash{\begin{tabular}[t]{l}$\tau_1$\end{tabular}}}}%
    \put(0.66222264,0.31053644){\color[rgb]{0,0,0}\makebox(0,0)[lt]{\lineheight{1.25}\smash{\begin{tabular}[t]{l}$\tau_2$\end{tabular}}}}%
    \put(0.90144544,0.31053644){\color[rgb]{0,0,0}\makebox(0,0)[lt]{\lineheight{1.25}\smash{\begin{tabular}[t]{l}$\tau_3$\end{tabular}}}}%
  \end{picture}%
\endgroup%

%% file: Strains.pdf_tex
\begingroup%
  \makeatletter%
  \providecommand\color[2][]{%
    \errmessage{(Inkscape) Color is used for the text in Inkscape, but the package 'color.sty' is not loaded}%
    \renewcommand\color[2][]{}%
  }%
  \providecommand\transparent[1]{%
    \errmessage{(Inkscape) Transparency is used (non-zero) for the text in Inkscape, but the package 'transparent.sty' is not loaded}%
    \renewcommand\transparent[1]{}%
  }%
  \providecommand\rotatebox[2]{#2}%
  \newcommand*\fsize{\dimexpr\f@size pt\relax}%
  \newcommand*\lineheight[1]{\fontsize{\fsize}{#1\fsize}\selectfont}%
  \ifx\svgwidth\undefined%
    \setlength{\unitlength}{420bp}%
    \ifx\svgscale\undefined%
      \relax%
    \else%
      \setlength{\unitlength}{\unitlength * \real{\svgscale}}%
    \fi%
  \else%
    \setlength{\unitlength}{\svgwidth}%
  \fi%
  \global\let\svgwidth\undefined%
  \global\let\svgscale\undefined%
  \makeatother%
  \begin{picture}(1,0.75)%
    \lineheight{1}%
    \setlength\tabcolsep{0pt}%
    \put(0,0){\includegraphics[width=\unitlength,page=1]{Strains.pdf}}%
    \put(0.13035714,0.04333332){\makebox(0,0)[t]{\lineheight{1.25}\smash{\begin{tabular}[t]{c}0\end{tabular}}}}%
    \put(0.25952375,0.04333332){\makebox(0,0)[t]{\lineheight{1.25}\smash{\begin{tabular}[t]{c}20\end{tabular}}}}%
    \put(0.38869054,0.04333332){\makebox(0,0)[t]{\lineheight{1.25}\smash{\begin{tabular}[t]{c}40\end{tabular}}}}%
    \put(0.51785714,0.04333332){\makebox(0,0)[t]{\lineheight{1.25}\smash{\begin{tabular}[t]{c}60\end{tabular}}}}%
    \put(0.64702375,0.04333332){\makebox(0,0)[t]{\lineheight{1.25}\smash{\begin{tabular}[t]{c}80\end{tabular}}}}%
    \put(0.77619054,0.04333332){\makebox(0,0)[t]{\lineheight{1.25}\smash{\begin{tabular}[t]{c}100\end{tabular}}}}%
    \put(0.90535714,0.04333332){\makebox(0,0)[t]{\lineheight{1.25}\smash{\begin{tabular}[t]{c}120\end{tabular}}}}%
    \put(0.5178575,0.00916654){\makebox(0,0)[t]{\lineheight{1.25}\smash{\begin{tabular}[t]{c}Time [s]\end{tabular}}}}%
    \put(0,0){\includegraphics[width=\unitlength,page=2]{Strains.pdf}}%
    \put(0.12083339,0.07571429){\makebox(0,0)[rt]{\lineheight{1.25}\smash{\begin{tabular}[t]{r}0\end{tabular}}}}%
    \put(0.12083339,0.13404768){\makebox(0,0)[rt]{\lineheight{1.25}\smash{\begin{tabular}[t]{r}2\end{tabular}}}}%
    \put(0.12083339,0.19238089){\makebox(0,0)[rt]{\lineheight{1.25}\smash{\begin{tabular}[t]{r}4\end{tabular}}}}%
    \put(0.12083339,0.25071429){\makebox(0,0)[rt]{\lineheight{1.25}\smash{\begin{tabular}[t]{r}6\end{tabular}}}}%
    \put(0.12083339,0.30904768){\makebox(0,0)[rt]{\lineheight{1.25}\smash{\begin{tabular}[t]{r}8\end{tabular}}}}%
    \put(0.12083339,0.36738089){\makebox(0,0)[rt]{\lineheight{1.25}\smash{\begin{tabular}[t]{r}10\end{tabular}}}}%
    \put(0.12083339,0.42571429){\makebox(0,0)[rt]{\lineheight{1.25}\smash{\begin{tabular}[t]{r}12\end{tabular}}}}%
    \put(0.12083339,0.48404768){\makebox(0,0)[rt]{\lineheight{1.25}\smash{\begin{tabular}[t]{r}14\end{tabular}}}}%
    \put(0.12083339,0.54238089){\makebox(0,0)[rt]{\lineheight{1.25}\smash{\begin{tabular}[t]{r}16\end{tabular}}}}%
    \put(0.12083339,0.60071429){\makebox(0,0)[rt]{\lineheight{1.25}\smash{\begin{tabular}[t]{r}18\end{tabular}}}}%
    \put(0.12083339,0.65904768){\makebox(0,0)[rt]{\lineheight{1.25}\smash{\begin{tabular}[t]{r}20\end{tabular}}}}%
    \put(0.04547625,0.38839321){\rotatebox{90}{\makebox(0,0)[t]{\lineheight{1.25}\smash{\begin{tabular}[t]{c}Strain [\%]\end{tabular}}}}}%
    \put(0,0){\includegraphics[width=\unitlength,page=3]{Strains.pdf}}%
    \put(0.77857143,0.18907645){\makebox(0,0)[lt]{\lineheight{1.25}\smash{\begin{tabular}[t]{l}1mm/s\end{tabular}}}}%
    \put(0,0){\includegraphics[width=\unitlength,page=4]{Strains.pdf}}%
    \put(0.77857143,0.1524693){\makebox(0,0)[lt]{\lineheight{1.25}\smash{\begin{tabular}[t]{l}2mm/s\end{tabular}}}}%
    \put(0,0){\includegraphics[width=\unitlength,page=5]{Strains.pdf}}%
    \put(0.77857143,0.11589286){\makebox(0,0)[lt]{\lineheight{1.25}\smash{\begin{tabular}[t]{l}10mm/s\end{tabular}}}}%
    \put(0,0){\includegraphics[width=\unitlength,page=6]{Strains.pdf}}%
  \end{picture}%
\endgroup%

%% file: Stress.pdf_tex
\begingroup%
  \makeatletter%
  \providecommand\color[2][]{%
    \errmessage{(Inkscape) Color is used for the text in Inkscape, but the package 'color.sty' is not loaded}%
    \renewcommand\color[2][]{}%
  }%
  \providecommand\transparent[1]{%
    \errmessage{(Inkscape) Transparency is used (non-zero) for the text in Inkscape, but the package 'transparent.sty' is not loaded}%
    \renewcommand\transparent[1]{}%
  }%
  \providecommand\rotatebox[2]{#2}%
  \newcommand*\fsize{\dimexpr\f@size pt\relax}%
  \newcommand*\lineheight[1]{\fontsize{\fsize}{#1\fsize}\selectfont}%
  \ifx\svgwidth\undefined%
    \setlength{\unitlength}{420bp}%
    \ifx\svgscale\undefined%
      \relax%
    \else%
      \setlength{\unitlength}{\unitlength * \real{\svgscale}}%
    \fi%
  \else%
    \setlength{\unitlength}{\svgwidth}%
  \fi%
  \global\let\svgwidth\undefined%
  \global\let\svgscale\undefined%
  \makeatother%
  \begin{picture}(1,0.75)%
    \lineheight{1}%
    \setlength\tabcolsep{0pt}%
    \put(0,0){\includegraphics[width=\unitlength,page=1]{Stress.pdf}}%
    \put(0.13035714,0.04333332){\makebox(0,0)[t]{\lineheight{1.25}\smash{\begin{tabular}[t]{c}0\end{tabular}}}}%
    \put(0.25952375,0.04333332){\makebox(0,0)[t]{\lineheight{1.25}\smash{\begin{tabular}[t]{c}20\end{tabular}}}}%
    \put(0.38869054,0.04333332){\makebox(0,0)[t]{\lineheight{1.25}\smash{\begin{tabular}[t]{c}40\end{tabular}}}}%
    \put(0.51785714,0.04333332){\makebox(0,0)[t]{\lineheight{1.25}\smash{\begin{tabular}[t]{c}60\end{tabular}}}}%
    \put(0.64702375,0.04333332){\makebox(0,0)[t]{\lineheight{1.25}\smash{\begin{tabular}[t]{c}80\end{tabular}}}}%
    \put(0.77619054,0.04333332){\makebox(0,0)[t]{\lineheight{1.25}\smash{\begin{tabular}[t]{c}100\end{tabular}}}}%
    \put(0.90535714,0.04333332){\makebox(0,0)[t]{\lineheight{1.25}\smash{\begin{tabular}[t]{c}120\end{tabular}}}}%
    \put(0.5178575,0.00916654){\makebox(0,0)[t]{\lineheight{1.25}\smash{\begin{tabular}[t]{c}Time [s]\end{tabular}}}}%
    \put(0,0){\includegraphics[width=\unitlength,page=2]{Stress.pdf}}%
    \put(0.12083339,0.07571429){\makebox(0,0)[rt]{\lineheight{1.25}\smash{\begin{tabular}[t]{r}0\end{tabular}}}}%
    \put(0.12083339,0.16321429){\makebox(0,0)[rt]{\lineheight{1.25}\smash{\begin{tabular}[t]{r}50\end{tabular}}}}%
    \put(0.12083339,0.25071429){\makebox(0,0)[rt]{\lineheight{1.25}\smash{\begin{tabular}[t]{r}100\end{tabular}}}}%
    \put(0.12083339,0.33821429){\makebox(0,0)[rt]{\lineheight{1.25}\smash{\begin{tabular}[t]{r}150\end{tabular}}}}%
    \put(0.12083339,0.42571429){\makebox(0,0)[rt]{\lineheight{1.25}\smash{\begin{tabular}[t]{r}200\end{tabular}}}}%
    \put(0.12083339,0.51321429){\makebox(0,0)[rt]{\lineheight{1.25}\smash{\begin{tabular}[t]{r}250\end{tabular}}}}%
    \put(0.12083339,0.60071429){\makebox(0,0)[rt]{\lineheight{1.25}\smash{\begin{tabular}[t]{r}300\end{tabular}}}}%
    \put(0.12083339,0.68821429){\makebox(0,0)[rt]{\lineheight{1.25}\smash{\begin{tabular}[t]{r}350\end{tabular}}}}%
    \put(0.04547625,0.38839321){\rotatebox{90}{\makebox(0,0)[t]{\lineheight{1.25}\smash{\begin{tabular}[t]{c}Stress [MPa]\end{tabular}}}}}%
    \put(0,0){\includegraphics[width=\unitlength,page=3]{Stress.pdf}}%
    \put(0.77857143,0.18907645){\makebox(0,0)[lt]{\lineheight{1.25}\smash{\begin{tabular}[t]{l}1mm/s\end{tabular}}}}%
    \put(0,0){\includegraphics[width=\unitlength,page=4]{Stress.pdf}}%
    \put(0.77857143,0.1524693){\makebox(0,0)[lt]{\lineheight{1.25}\smash{\begin{tabular}[t]{l}2mm/s\end{tabular}}}}%
    \put(0,0){\includegraphics[width=\unitlength,page=5]{Stress.pdf}}%
    \put(0.77857143,0.11589286){\makebox(0,0)[lt]{\lineheight{1.25}\smash{\begin{tabular}[t]{l}10mm/s\end{tabular}}}}%
    \put(0,0){\includegraphics[width=\unitlength,page=6]{Stress.pdf}}%
  \end{picture}%
\endgroup%

%% file: product_E_tau.pdf_tex
\begingroup%
  \makeatletter%
  \providecommand\color[2][]{%
    \errmessage{(Inkscape) Color is used for the text in Inkscape, but the package 'color.sty' is not loaded}%
    \renewcommand\color[2][]{}%
  }%
  \providecommand\transparent[1]{%
    \errmessage{(Inkscape) Transparency is used (non-zero) for the text in Inkscape, but the package 'transparent.sty' is not loaded}%
    \renewcommand\transparent[1]{}%
  }%
  \providecommand\rotatebox[2]{#2}%
  \newcommand*\fsize{\dimexpr\f@size pt\relax}%
  \newcommand*\lineheight[1]{\fontsize{\fsize}{#1\fsize}\selectfont}%
  \ifx\svgwidth\undefined%
    \setlength{\unitlength}{891bp}%
    \ifx\svgscale\undefined%
      \relax%
    \else%
      \setlength{\unitlength}{\unitlength * \real{\svgscale}}%
    \fi%
  \else%
    \setlength{\unitlength}{\svgwidth}%
  \fi%
  \global\let\svgwidth\undefined%
  \global\let\svgscale\undefined%
  \makeatother%
  \begin{picture}(1,0.60016835)%
    \lineheight{1}%
    \setlength\tabcolsep{0pt}%
    \put(0,0){\includegraphics[width=\unitlength,page=1]{product_E_tau.pdf}}%
    \put(0.11674039,0.03409619){\makebox(0,0)[lt]{\lineheight{1.25}\smash{\begin{tabular}[t]{l}0\end{tabular}}}}%
    \put(0.23484848,0.03720539){\makebox(0,0)[lt]{\lineheight{1.25}\smash{\begin{tabular}[t]{l}10\end{tabular}}}}%
    \put(0.34553872,0.03720539){\makebox(0,0)[lt]{\lineheight{1.25}\smash{\begin{tabular}[t]{l}20\end{tabular}}}}%
    \put(0.45622896,0.03720539){\makebox(0,0)[lt]{\lineheight{1.25}\smash{\begin{tabular}[t]{l}30\end{tabular}}}}%
    \put(0.56691919,0.03720539){\makebox(0,0)[lt]{\lineheight{1.25}\smash{\begin{tabular}[t]{l}40\end{tabular}}}}%
    \put(0.67760943,0.03720539){\makebox(0,0)[lt]{\lineheight{1.25}\smash{\begin{tabular}[t]{l}50\end{tabular}}}}%
    \put(0.78829966,0.03720539){\makebox(0,0)[lt]{\lineheight{1.25}\smash{\begin{tabular}[t]{l}60\end{tabular}}}}%
    \put(0.8989899,0.03720539){\makebox(0,0)[lt]{\lineheight{1.25}\smash{\begin{tabular}[t]{l}70\end{tabular}}}}%
    \put(0,0){\includegraphics[width=\unitlength,page=2]{product_E_tau.pdf}}%
    \put(0.08829166,0.11852525){\makebox(0,0)[lt]{\lineheight{1.25}\smash{\begin{tabular}[t]{l}0.1\end{tabular}}}}%
    \put(0.08829166,0.17965741){\makebox(0,0)[lt]{\lineheight{1.25}\smash{\begin{tabular}[t]{l}0.2\end{tabular}}}}%
    \put(0.08829166,0.24078956){\makebox(0,0)[lt]{\lineheight{1.25}\smash{\begin{tabular}[t]{l}0.3\end{tabular}}}}%
    \put(0.08829166,0.30192172){\makebox(0,0)[lt]{\lineheight{1.25}\smash{\begin{tabular}[t]{l}0.4\end{tabular}}}}%
    \put(0.08829166,0.36305387){\makebox(0,0)[lt]{\lineheight{1.25}\smash{\begin{tabular}[t]{l}0.5\end{tabular}}}}%
    \put(0.08829166,0.42418603){\makebox(0,0)[lt]{\lineheight{1.25}\smash{\begin{tabular}[t]{l}0.6\end{tabular}}}}%
    \put(0.08829166,0.48531818){\makebox(0,0)[lt]{\lineheight{1.25}\smash{\begin{tabular}[t]{l}0.7\end{tabular}}}}%
    \put(0.08829166,0.54645034){\makebox(0,0)[lt]{\lineheight{1.25}\smash{\begin{tabular}[t]{l}0.8\end{tabular}}}}%
    \put(0,0){\includegraphics[width=\unitlength,page=3]{product_E_tau.pdf}}%
    \put(0.73545868,0.51463047){\color[rgb]{0,0,0}\makebox(0,0)[lt]{\lineheight{1.25}\smash{\begin{tabular}[t]{l}exact\end{tabular}}}}%
    \put(0.73538901,0.49069412){\color[rgb]{0,0,0}\makebox(0,0)[lt]{\lineheight{1.25}\smash{\begin{tabular}[t]{l}computed\\\end{tabular}}}}%
    \put(0.06544161,0.2200071){\color[rgb]{0,0,0}\rotatebox{90}{\makebox(0,0)[lt]{\lineheight{1.25}\smash{\begin{tabular}[t]{l}relaxation time $\tau_1$ [s]\end{tabular}}}}}%
    \put(0.45516983,0.0101399){\color[rgb]{0,0,0}\makebox(0,0)[lt]{\lineheight{1.25}\smash{\begin{tabular}[t]{l}stiffness $\mu_1$ [MPa]\end{tabular}}}}%
  \end{picture}%
\endgroup%

%% file: 10mmps.pdf_tex
\begingroup%
  \makeatletter%
  \providecommand\color[2][]{%
    \errmessage{(Inkscape) Color is used for the text in Inkscape, but the package 'color.sty' is not loaded}%
    \renewcommand\color[2][]{}%
  }%
  \providecommand\transparent[1]{%
    \errmessage{(Inkscape) Transparency is used (non-zero) for the text in Inkscape, but the package 'transparent.sty' is not loaded}%
    \renewcommand\transparent[1]{}%
  }%
  \providecommand\rotatebox[2]{#2}%
  \newcommand*\fsize{\dimexpr\f@size pt\relax}%
  \newcommand*\lineheight[1]{\fontsize{\fsize}{#1\fsize}\selectfont}%
  \ifx\svgwidth\undefined%
    \setlength{\unitlength}{630bp}%
    \ifx\svgscale\undefined%
      \relax%
    \else%
      \setlength{\unitlength}{\unitlength * \real{\svgscale}}%
    \fi%
  \else%
    \setlength{\unitlength}{\svgwidth}%
  \fi%
  \global\let\svgwidth\undefined%
  \global\let\svgscale\undefined%
  \makeatother%
  \begin{picture}(1,0.75)%
    \lineheight{1}%
    \setlength\tabcolsep{0pt}%
    \put(0,0){\includegraphics[width=\unitlength,page=1]{10mmps.pdf}}%
    \put(0.12261905,0.04761905){\makebox(0,0)[lt]{\lineheight{1.25}\smash{\begin{tabular}[t]{l}0\end{tabular}}}}%
    \put(0.17428571,0.04761905){\makebox(0,0)[lt]{\lineheight{1.25}\smash{\begin{tabular}[t]{l}2\end{tabular}}}}%
    \put(0.22595238,0.04761905){\makebox(0,0)[lt]{\lineheight{1.25}\smash{\begin{tabular}[t]{l}4\end{tabular}}}}%
    \put(0.27761905,0.04761905){\makebox(0,0)[lt]{\lineheight{1.25}\smash{\begin{tabular}[t]{l}6\end{tabular}}}}%
    \put(0.32928571,0.04761905){\makebox(0,0)[lt]{\lineheight{1.25}\smash{\begin{tabular}[t]{l}8\end{tabular}}}}%
    \put(0.37440476,0.04761905){\makebox(0,0)[lt]{\lineheight{1.25}\smash{\begin{tabular}[t]{l}10\end{tabular}}}}%
    \put(0.42607143,0.04761905){\makebox(0,0)[lt]{\lineheight{1.25}\smash{\begin{tabular}[t]{l}12\end{tabular}}}}%
    \put(0.4777381,0.04761905){\makebox(0,0)[lt]{\lineheight{1.25}\smash{\begin{tabular}[t]{l}14\end{tabular}}}}%
    \put(0.52940476,0.04761905){\makebox(0,0)[lt]{\lineheight{1.25}\smash{\begin{tabular}[t]{l}16\end{tabular}}}}%
    \put(0.58107143,0.04761905){\makebox(0,0)[lt]{\lineheight{1.25}\smash{\begin{tabular}[t]{l}18\end{tabular}}}}%
    \put(0.6327381,0.04761905){\makebox(0,0)[lt]{\lineheight{1.25}\smash{\begin{tabular}[t]{l}20\end{tabular}}}}%
    \put(0.68440476,0.04761905){\makebox(0,0)[lt]{\lineheight{1.25}\smash{\begin{tabular}[t]{l}22\end{tabular}}}}%
    \put(0.73607143,0.04761905){\makebox(0,0)[lt]{\lineheight{1.25}\smash{\begin{tabular}[t]{l}24\end{tabular}}}}%
    \put(0.7877381,0.04761905){\makebox(0,0)[lt]{\lineheight{1.25}\smash{\begin{tabular}[t]{l}26\end{tabular}}}}%
    \put(0.83940476,0.04761905){\makebox(0,0)[lt]{\lineheight{1.25}\smash{\begin{tabular}[t]{l}28\end{tabular}}}}%
    \put(0.89107143,0.04761905){\makebox(0,0)[lt]{\lineheight{1.25}\smash{\begin{tabular}[t]{l}30\end{tabular}}}}%
    \put(0.45168976,0.0143179){\makebox(0,0)[lt]{\lineheight{1.25}\smash{\begin{tabular}[t]{l}Time [s]\end{tabular}}}}%
    \put(0,0){\includegraphics[width=\unitlength,page=2]{10mmps.pdf}}%
    \put(0.10595238,0.07261905){\makebox(0,0)[lt]{\lineheight{1.25}\smash{\begin{tabular}[t]{l}0\end{tabular}}}}%
    \put(0.09285714,0.13089571){\makebox(0,0)[lt]{\lineheight{1.25}\smash{\begin{tabular}[t]{l}20\end{tabular}}}}%
    \put(0.09285714,0.18917238){\makebox(0,0)[lt]{\lineheight{1.25}\smash{\begin{tabular}[t]{l}40\end{tabular}}}}%
    \put(0.09285714,0.24744893){\makebox(0,0)[lt]{\lineheight{1.25}\smash{\begin{tabular}[t]{l}60\end{tabular}}}}%
    \put(0.09285714,0.3057256){\makebox(0,0)[lt]{\lineheight{1.25}\smash{\begin{tabular}[t]{l}80\end{tabular}}}}%
    \put(0.0797619,0.36400226){\makebox(0,0)[lt]{\lineheight{1.25}\smash{\begin{tabular}[t]{l}100\end{tabular}}}}%
    \put(0.0797619,0.42227893){\makebox(0,0)[lt]{\lineheight{1.25}\smash{\begin{tabular}[t]{l}120\end{tabular}}}}%
    \put(0.0797619,0.4805556){\makebox(0,0)[lt]{\lineheight{1.25}\smash{\begin{tabular}[t]{l}140\end{tabular}}}}%
    \put(0.0797619,0.53883214){\makebox(0,0)[lt]{\lineheight{1.25}\smash{\begin{tabular}[t]{l}160\end{tabular}}}}%
    \put(0.0797619,0.59710881){\makebox(0,0)[lt]{\lineheight{1.25}\smash{\begin{tabular}[t]{l}180\end{tabular}}}}%
    \put(0.0797619,0.65538548){\makebox(0,0)[lt]{\lineheight{1.25}\smash{\begin{tabular}[t]{l}200\end{tabular}}}}%
    \put(0,0){\includegraphics[width=\unitlength,page=3]{10mmps.pdf}}%
    \put(0.75389671,0.60630931){\makebox(0,0)[lt]{\lineheight{1.25}\smash{\begin{tabular}[t]{l}$\sigma_0(t)$\end{tabular}}}}%
    \put(0.75389671,0.54596514){\makebox(0,0)[lt]{\lineheight{1.25}\smash{\begin{tabular}[t]{l}$\sigma_1(t)$\end{tabular}}}}%
    \put(0.75389671,0.4856211){\makebox(0,0)[lt]{\lineheight{1.25}\smash{\begin{tabular}[t]{l}$\sigma_2(t)$\end{tabular}}}}%
    \put(0.75389671,0.42527693){\makebox(0,0)[lt]{\lineheight{1.25}\smash{\begin{tabular}[t]{l}$\sigma_3(t)$\end{tabular}}}}%
    \put(0.06179518,0.31159443){\color[rgb]{0,0,0}\rotatebox{90}{\makebox(0,0)[lt]{\lineheight{1.25}\smash{\begin{tabular}[t]{l}Stress [MPa]\end{tabular}}}}}%
  \end{picture}%
\endgroup%

%% file: 1mmps.pdf_tex
\begingroup%
  \makeatletter%
  \providecommand\color[2][]{%
    \errmessage{(Inkscape) Color is used for the text in Inkscape, but the package 'color.sty' is not loaded}%
    \renewcommand\color[2][]{}%
  }%
  \providecommand\transparent[1]{%
    \errmessage{(Inkscape) Transparency is used (non-zero) for the text in Inkscape, but the package 'transparent.sty' is not loaded}%
    \renewcommand\transparent[1]{}%
  }%
  \providecommand\rotatebox[2]{#2}%
  \newcommand*\fsize{\dimexpr\f@size pt\relax}%
  \newcommand*\lineheight[1]{\fontsize{\fsize}{#1\fsize}\selectfont}%
  \ifx\svgwidth\undefined%
    \setlength{\unitlength}{630bp}%
    \ifx\svgscale\undefined%
      \relax%
    \else%
      \setlength{\unitlength}{\unitlength * \real{\svgscale}}%
    \fi%
  \else%
    \setlength{\unitlength}{\svgwidth}%
  \fi%
  \global\let\svgwidth\undefined%
  \global\let\svgscale\undefined%
  \makeatother%
  \begin{picture}(1,0.75)%
    \lineheight{1}%
    \setlength\tabcolsep{0pt}%
    \put(0,0){\includegraphics[width=\unitlength,page=1]{1mmps.pdf}}%
    \put(0.12261905,0.04761905){\makebox(0,0)[lt]{\lineheight{1.25}\smash{\begin{tabular}[t]{l}0\end{tabular}}}}%
    \put(0.17428571,0.04761905){\makebox(0,0)[lt]{\lineheight{1.25}\smash{\begin{tabular}[t]{l}2\end{tabular}}}}%
    \put(0.22595238,0.04761905){\makebox(0,0)[lt]{\lineheight{1.25}\smash{\begin{tabular}[t]{l}4\end{tabular}}}}%
    \put(0.27761905,0.04761905){\makebox(0,0)[lt]{\lineheight{1.25}\smash{\begin{tabular}[t]{l}6\end{tabular}}}}%
    \put(0.32928571,0.04761905){\makebox(0,0)[lt]{\lineheight{1.25}\smash{\begin{tabular}[t]{l}8\end{tabular}}}}%
    \put(0.37440476,0.04761905){\makebox(0,0)[lt]{\lineheight{1.25}\smash{\begin{tabular}[t]{l}10\end{tabular}}}}%
    \put(0.42607143,0.04761905){\makebox(0,0)[lt]{\lineheight{1.25}\smash{\begin{tabular}[t]{l}12\end{tabular}}}}%
    \put(0.4777381,0.04761905){\makebox(0,0)[lt]{\lineheight{1.25}\smash{\begin{tabular}[t]{l}14\end{tabular}}}}%
    \put(0.52940476,0.04761905){\makebox(0,0)[lt]{\lineheight{1.25}\smash{\begin{tabular}[t]{l}16\end{tabular}}}}%
    \put(0.58107143,0.04761905){\makebox(0,0)[lt]{\lineheight{1.25}\smash{\begin{tabular}[t]{l}18\end{tabular}}}}%
    \put(0.6327381,0.04761905){\makebox(0,0)[lt]{\lineheight{1.25}\smash{\begin{tabular}[t]{l}20\end{tabular}}}}%
    \put(0.68440476,0.04761905){\makebox(0,0)[lt]{\lineheight{1.25}\smash{\begin{tabular}[t]{l}22\end{tabular}}}}%
    \put(0.73607143,0.04761905){\makebox(0,0)[lt]{\lineheight{1.25}\smash{\begin{tabular}[t]{l}24\end{tabular}}}}%
    \put(0.7877381,0.04761905){\makebox(0,0)[lt]{\lineheight{1.25}\smash{\begin{tabular}[t]{l}26\end{tabular}}}}%
    \put(0.83940476,0.04761905){\makebox(0,0)[lt]{\lineheight{1.25}\smash{\begin{tabular}[t]{l}28\end{tabular}}}}%
    \put(0.89107143,0.04761905){\makebox(0,0)[lt]{\lineheight{1.25}\smash{\begin{tabular}[t]{l}30\end{tabular}}}}%
    \put(0.45404574,0.01190476){\makebox(0,0)[lt]{\lineheight{1.25}\smash{\begin{tabular}[t]{l}Time [s]\end{tabular}}}}%
    \put(0,0){\includegraphics[width=\unitlength,page=2]{1mmps.pdf}}%
    \put(0.10595238,0.07261905){\makebox(0,0)[lt]{\lineheight{1.25}\smash{\begin{tabular}[t]{l}0\end{tabular}}}}%
    \put(0.09285714,0.13089571){\makebox(0,0)[lt]{\lineheight{1.25}\smash{\begin{tabular}[t]{l}20\end{tabular}}}}%
    \put(0.09285714,0.18917238){\makebox(0,0)[lt]{\lineheight{1.25}\smash{\begin{tabular}[t]{l}40\end{tabular}}}}%
    \put(0.09285714,0.24744893){\makebox(0,0)[lt]{\lineheight{1.25}\smash{\begin{tabular}[t]{l}60\end{tabular}}}}%
    \put(0.09285714,0.3057256){\makebox(0,0)[lt]{\lineheight{1.25}\smash{\begin{tabular}[t]{l}80\end{tabular}}}}%
    \put(0.0797619,0.36400226){\makebox(0,0)[lt]{\lineheight{1.25}\smash{\begin{tabular}[t]{l}100\end{tabular}}}}%
    \put(0.0797619,0.42227893){\makebox(0,0)[lt]{\lineheight{1.25}\smash{\begin{tabular}[t]{l}120\end{tabular}}}}%
    \put(0.0797619,0.4805556){\makebox(0,0)[lt]{\lineheight{1.25}\smash{\begin{tabular}[t]{l}140\end{tabular}}}}%
    \put(0.0797619,0.53883214){\makebox(0,0)[lt]{\lineheight{1.25}\smash{\begin{tabular}[t]{l}160\end{tabular}}}}%
    \put(0.0797619,0.59710881){\makebox(0,0)[lt]{\lineheight{1.25}\smash{\begin{tabular}[t]{l}180\end{tabular}}}}%
    \put(0.0797619,0.65538548){\makebox(0,0)[lt]{\lineheight{1.25}\smash{\begin{tabular}[t]{l}200\end{tabular}}}}%
    \put(0,0){\includegraphics[width=\unitlength,page=3]{1mmps.pdf}}%
    \put(0.75274007,0.60599256){\makebox(0,0)[lt]{\lineheight{1.25}\smash{\begin{tabular}[t]{l}$\sigma_0(t)$\end{tabular}}}}%
    \put(0,0){\includegraphics[width=\unitlength,page=4]{1mmps.pdf}}%
    \put(0.75274007,0.54564839){\makebox(0,0)[lt]{\lineheight{1.25}\smash{\begin{tabular}[t]{l}$\sigma_1(t)$\end{tabular}}}}%
    \put(0,0){\includegraphics[width=\unitlength,page=5]{1mmps.pdf}}%
    \put(0.75274007,0.48530435){\makebox(0,0)[lt]{\lineheight{1.25}\smash{\begin{tabular}[t]{l}$\sigma_2(t)$\end{tabular}}}}%
    \put(0,0){\includegraphics[width=\unitlength,page=6]{1mmps.pdf}}%
    \put(0.75274007,0.42496018){\makebox(0,0)[lt]{\lineheight{1.25}\smash{\begin{tabular}[t]{l}$\sigma_3(t)$\end{tabular}}}}%
    \put(0,0){\includegraphics[width=\unitlength,page=7]{1mmps.pdf}}%
    \put(0.06002813,0.3101892){\color[rgb]{0,0,0}\rotatebox{90}{\makebox(0,0)[lt]{\lineheight{1.25}\smash{\begin{tabular}[t]{l}Stress [MPa]\end{tabular}}}}}%
  \end{picture}%
\endgroup%

%% file: boxplot_smallest.pdf_tex
\begingroup%
  \makeatletter%
  \providecommand\color[2][]{%
    \errmessage{(Inkscape) Color is used for the text in Inkscape, but the package 'color.sty' is not loaded}%
    \renewcommand\color[2][]{}%
  }%
  \providecommand\transparent[1]{%
    \errmessage{(Inkscape) Transparency is used (non-zero) for the text in Inkscape, but the package 'transparent.sty' is not loaded}%
    \renewcommand\transparent[1]{}%
  }%
  \providecommand\rotatebox[2]{#2}%
  \newcommand*\fsize{\dimexpr\f@size pt\relax}%
  \newcommand*\lineheight[1]{\fontsize{\fsize}{#1\fsize}\selectfont}%
  \ifx\svgwidth\undefined%
    \setlength{\unitlength}{420bp}%
    \ifx\svgscale\undefined%
      \relax%
    \else%
      \setlength{\unitlength}{\unitlength * \real{\svgscale}}%
    \fi%
  \else%
    \setlength{\unitlength}{\svgwidth}%
  \fi%
  \global\let\svgwidth\undefined%
  \global\let\svgscale\undefined%
  \makeatother%
  \begin{picture}(1,0.75)%
    \lineheight{1}%
    \setlength\tabcolsep{0pt}%
    \put(0,0){\includegraphics[width=\unitlength,page=1]{boxplot_smallest.pdf}}%
    \put(0.18708393,0.03035714){\makebox(0,0)[lt]{\lineheight{1.25}\smash{\begin{tabular}[t]{l}(a)\end{tabular}}}}%
    \put(0.47035768,0.03035714){\makebox(0,0)[lt]{\lineheight{1.25}\smash{\begin{tabular}[t]{l}(b)\end{tabular}}}}%
    \put(0.75541714,0.03035714){\makebox(0,0)[lt]{\lineheight{1.25}\smash{\begin{tabular}[t]{l}(c)\end{tabular}}}}%
    \put(0,0){\includegraphics[width=\unitlength,page=2]{boxplot_smallest.pdf}}%
    \put(0.08571429,0.0827625){\makebox(0,0)[lt]{\lineheight{1.25}\smash{\begin{tabular}[t]{l}0\end{tabular}}}}%
    \put(0.07142857,0.16050179){\makebox(0,0)[lt]{\lineheight{1.25}\smash{\begin{tabular}[t]{l}0.1\end{tabular}}}}%
    \put(0.07142857,0.23824107){\makebox(0,0)[lt]{\lineheight{1.25}\smash{\begin{tabular}[t]{l}0.2\end{tabular}}}}%
    \put(0.07142857,0.31598054){\makebox(0,0)[lt]{\lineheight{1.25}\smash{\begin{tabular}[t]{l}0.3\end{tabular}}}}%
    \put(0.07142857,0.39371982){\makebox(0,0)[lt]{\lineheight{1.25}\smash{\begin{tabular}[t]{l}0.4\end{tabular}}}}%
    \put(0.07142857,0.47145911){\makebox(0,0)[lt]{\lineheight{1.25}\smash{\begin{tabular}[t]{l}0.5\end{tabular}}}}%
    \put(0.07142857,0.54919839){\makebox(0,0)[lt]{\lineheight{1.25}\smash{\begin{tabular}[t]{l}0.6\end{tabular}}}}%
    \put(0.07142857,0.62693768){\makebox(0,0)[lt]{\lineheight{1.25}\smash{\begin{tabular}[t]{l}0.7\end{tabular}}}}%
    \put(0,0){\includegraphics[width=\unitlength,page=3]{boxplot_smallest.pdf}}%
    \put(0.04084945,0.23160346){\color[rgb]{0,0,0}\rotatebox{90}{\makebox(0,0)[lt]{\lineheight{1.25}\smash{\begin{tabular}[t]{l}relaxation time $\tau_1$\end{tabular}}}}}%
  \end{picture}%
\endgroup%

%% file: boxplot_largest.pdf_tex
\begingroup%
  \makeatletter%
  \providecommand\color[2][]{%
    \errmessage{(Inkscape) Color is used for the text in Inkscape, but the package 'color.sty' is not loaded}%
    \renewcommand\color[2][]{}%
  }%
  \providecommand\transparent[1]{%
    \errmessage{(Inkscape) Transparency is used (non-zero) for the text in Inkscape, but the package 'transparent.sty' is not loaded}%
    \renewcommand\transparent[1]{}%
  }%
  \providecommand\rotatebox[2]{#2}%
  \newcommand*\fsize{\dimexpr\f@size pt\relax}%
  \newcommand*\lineheight[1]{\fontsize{\fsize}{#1\fsize}\selectfont}%
  \ifx\svgwidth\undefined%
    \setlength{\unitlength}{420bp}%
    \ifx\svgscale\undefined%
      \relax%
    \else%
      \setlength{\unitlength}{\unitlength * \real{\svgscale}}%
    \fi%
  \else%
    \setlength{\unitlength}{\svgwidth}%
  \fi%
  \global\let\svgwidth\undefined%
  \global\let\svgscale\undefined%
  \makeatother%
  \begin{picture}(1,0.75)%
    \lineheight{1}%
    \setlength\tabcolsep{0pt}%
    \put(0,0){\includegraphics[width=\unitlength,page=1]{boxplot_largest.pdf}}%
    \put(0.24077321,0.03928571){\makebox(0,0)[lt]{\lineheight{1.25}\smash{\begin{tabular}[t]{l}(a)\end{tabular}}}}%
    \put(0.49821375,0.03928571){\makebox(0,0)[lt]{\lineheight{1.25}\smash{\begin{tabular}[t]{l}(b)\end{tabular}}}}%
    \put(0.75744,0.03928571){\makebox(0,0)[lt]{\lineheight{1.25}\smash{\begin{tabular}[t]{l}(c)\end{tabular}}}}%
    \put(0,0){\includegraphics[width=\unitlength,page=2]{boxplot_largest.pdf}}%
    \put(0.08928571,0.17756768){\makebox(0,0)[lt]{\lineheight{1.25}\smash{\begin{tabular}[t]{l}25\end{tabular}}}}%
    \put(0.08928571,0.28568571){\makebox(0,0)[lt]{\lineheight{1.25}\smash{\begin{tabular}[t]{l}30\end{tabular}}}}%
    \put(0.08928571,0.39380393){\makebox(0,0)[lt]{\lineheight{1.25}\smash{\begin{tabular}[t]{l}35\end{tabular}}}}%
    \put(0.08928571,0.50192214){\makebox(0,0)[lt]{\lineheight{1.25}\smash{\begin{tabular}[t]{l}40\end{tabular}}}}%
    \put(0.08928571,0.61004018){\makebox(0,0)[lt]{\lineheight{1.25}\smash{\begin{tabular}[t]{l}45\end{tabular}}}}%
    \put(0,0){\includegraphics[width=\unitlength,page=3]{boxplot_largest.pdf}}%
    \put(0.05173022,0.25561442){\color[rgb]{0,0,0}\rotatebox{90}{\makebox(0,0)[lt]{\lineheight{1.25}\smash{\begin{tabular}[t]{l}relaxation time $\tau_3$\end{tabular}}}}}%
  \end{picture}%
\endgroup%

%% file: shortened_data_10mm_mu0.pdf_tex
\begingroup%
  \makeatletter%
  \providecommand\color[2][]{%
    \errmessage{(Inkscape) Color is used for the text in Inkscape, but the package 'color.sty' is not loaded}%
    \renewcommand\color[2][]{}%
  }%
  \providecommand\transparent[1]{%
    \errmessage{(Inkscape) Transparency is used (non-zero) for the text in Inkscape, but the package 'transparent.sty' is not loaded}%
    \renewcommand\transparent[1]{}%
  }%
  \providecommand\rotatebox[2]{#2}%
  \newcommand*\fsize{\dimexpr\f@size pt\relax}%
  \newcommand*\lineheight[1]{\fontsize{\fsize}{#1\fsize}\selectfont}%
  \ifx\svgwidth\undefined%
    \setlength{\unitlength}{549bp}%
    \ifx\svgscale\undefined%
      \relax%
    \else%
      \setlength{\unitlength}{\unitlength * \real{\svgscale}}%
    \fi%
  \else%
    \setlength{\unitlength}{\svgwidth}%
  \fi%
  \global\let\svgwidth\undefined%
  \global\let\svgscale\undefined%
  \makeatother%
  \begin{picture}(1,0.73224044)%
    \lineheight{1}%
    \setlength\tabcolsep{0pt}%
    \put(0,0){\includegraphics[width=\unitlength,page=1]{shortened_data_10mm_mu0.pdf}}%
    \put(0.07980554,0.07716143){\makebox(0,0)[lt]{\lineheight{1.25}\smash{\begin{tabular}[t]{l}1\end{tabular}}}}%
    \put(0.07980554,0.13686089){\makebox(0,0)[lt]{\lineheight{1.25}\smash{\begin{tabular}[t]{l}2\end{tabular}}}}%
    \put(0.07980554,0.19656034){\makebox(0,0)[lt]{\lineheight{1.25}\smash{\begin{tabular}[t]{l}3\end{tabular}}}}%
    \put(0.07980554,0.2562598){\makebox(0,0)[lt]{\lineheight{1.25}\smash{\begin{tabular}[t]{l}4\end{tabular}}}}%
    \put(0.07980554,0.31595925){\makebox(0,0)[lt]{\lineheight{1.25}\smash{\begin{tabular}[t]{l}5\end{tabular}}}}%
    \put(0.07980554,0.3756587){\makebox(0,0)[lt]{\lineheight{1.25}\smash{\begin{tabular}[t]{l}6\end{tabular}}}}%
    \put(0.07980554,0.43535816){\makebox(0,0)[lt]{\lineheight{1.25}\smash{\begin{tabular}[t]{l}7\end{tabular}}}}%
    \put(0.07980554,0.49505761){\makebox(0,0)[lt]{\lineheight{1.25}\smash{\begin{tabular}[t]{l}8\end{tabular}}}}%
    \put(0.07980554,0.55475706){\makebox(0,0)[lt]{\lineheight{1.25}\smash{\begin{tabular}[t]{l}9\end{tabular}}}}%
    \put(0.07160882,0.61445652){\makebox(0,0)[lt]{\lineheight{1.25}\smash{\begin{tabular}[t]{l}10\end{tabular}}}}%
    \put(0.07160882,0.67415597){\makebox(0,0)[lt]{\lineheight{1.25}\smash{\begin{tabular}[t]{l}11\end{tabular}}}}%
    \put(0,0){\includegraphics[width=\unitlength,page=2]{shortened_data_10mm_mu0.pdf}}%
    \put(0.11588809,0.03832672){\makebox(0,0)[lt]{\lineheight{1.25}\smash{\begin{tabular}[t]{l}100\end{tabular}}}}%
    \put(0.21686875,0.03832672){\makebox(0,0)[lt]{\lineheight{1.25}\smash{\begin{tabular}[t]{l}90\end{tabular}}}}%
    \put(0.31227201,0.03832672){\makebox(0,0)[lt]{\lineheight{1.25}\smash{\begin{tabular}[t]{l}80\end{tabular}}}}%
    \put(0.40973756,0.03832672){\makebox(0,0)[lt]{\lineheight{1.25}\smash{\begin{tabular}[t]{l}70\end{tabular}}}}%
    \put(0.50680906,0.03832672){\makebox(0,0)[lt]{\lineheight{1.25}\smash{\begin{tabular}[t]{l}60\end{tabular}}}}%
    \put(0.60401689,0.03832672){\makebox(0,0)[lt]{\lineheight{1.25}\smash{\begin{tabular}[t]{l}50\end{tabular}}}}%
    \put(0.70128328,0.03832672){\makebox(0,0)[lt]{\lineheight{1.25}\smash{\begin{tabular}[t]{l}40\end{tabular}}}}%
    \put(0.8002181,0.03832672){\makebox(0,0)[lt]{\lineheight{1.25}\smash{\begin{tabular}[t]{l}30\end{tabular}}}}%
    \put(0.31161372,-0.01521935){\color[rgb]{0,0,0}\makebox(0,0)[lt]{\lineheight{1.25}\smash{\begin{tabular}[t]{l}duration time T [s]\end{tabular}}}}%
    \put(0.89704097,0.03911746){\makebox(0,0)[lt]{\lineheight{1.25}\smash{\begin{tabular}[t]{l}20\end{tabular}}}}%
    \put(0.04672377,0.14920027){\color[rgb]{0,0,0}\rotatebox{90.038036}{\makebox(0,0)[lt]{\lineheight{1.25}\smash{\begin{tabular}[t]{l}stiffness $\mu$ [MPa]\end{tabular}}}}}%
    \put(0,0){\includegraphics[width=\unitlength,page=3]{shortened_data_10mm_mu0.pdf}}%
    \put(0.6696114,0.15720939){\color[rgb]{0,0,0}\makebox(0,0)[lt]{\lineheight{1.25}\smash{\begin{tabular}[t]{l}exact\end{tabular}}}}%
    \put(0.66449996,0.10939489){\color[rgb]{0,0,0}\makebox(0,0)[lt]{\lineheight{1.25}\smash{\begin{tabular}[t]{l}computed\end{tabular}}}}%
  \end{picture}%
\endgroup%

%% file: shortened_data_1mm_mu0.pdf_tex
\begingroup%
  \makeatletter%
  \providecommand\color[2][]{%
    \errmessage{(Inkscape) Color is used for the text in Inkscape, but the package 'color.sty' is not loaded}%
    \renewcommand\color[2][]{}%
  }%
  \providecommand\transparent[1]{%
    \errmessage{(Inkscape) Transparency is used (non-zero) for the text in Inkscape, but the package 'transparent.sty' is not loaded}%
    \renewcommand\transparent[1]{}%
  }%
  \providecommand\rotatebox[2]{#2}%
  \newcommand*\fsize{\dimexpr\f@size pt\relax}%
  \newcommand*\lineheight[1]{\fontsize{\fsize}{#1\fsize}\selectfont}%
  \ifx\svgwidth\undefined%
    \setlength{\unitlength}{420bp}%
    \ifx\svgscale\undefined%
      \relax%
    \else%
      \setlength{\unitlength}{\unitlength * \real{\svgscale}}%
    \fi%
  \else%
    \setlength{\unitlength}{\svgwidth}%
  \fi%
  \global\let\svgwidth\undefined%
  \global\let\svgscale\undefined%
  \makeatother%
  \begin{picture}(1,0.75)%
    \lineheight{1}%
    \setlength\tabcolsep{0pt}%
    \put(0,0){\includegraphics[width=\unitlength,page=1]{shortened_data_1mm_mu0.pdf}}%
    \put(0.3142462,-0.02487651){\color[rgb]{0,0,0}\makebox(0,0)[lt]{\lineheight{1.25}\smash{\begin{tabular}[t]{l}duration time T [s]\end{tabular}}}}%
    \put(0.06611827,0.15238026){\color[rgb]{0,0,0}\rotatebox{90.038036}{\makebox(0,0)[lt]{\lineheight{1.25}\smash{\begin{tabular}[t]{l}stiffness $\mu$ [MPa]\end{tabular}}}}}%
    \put(0,0){\includegraphics[width=\unitlength,page=2]{shortened_data_1mm_mu0.pdf}}%
    \put(0.23450168,0.17409483){\color[rgb]{0,0,0}\makebox(0,0)[lt]{\lineheight{1.25}\smash{\begin{tabular}[t]{l}exact\end{tabular}}}}%
    \put(0.2278203,0.11159445){\color[rgb]{0,0,0}\makebox(0,0)[lt]{\lineheight{1.25}\smash{\begin{tabular}[t]{l}computed\end{tabular}}}}%
    \put(0.12223959,0.02995104){\makebox(0,0)[lt]{\lineheight{1.25}\smash{\begin{tabular}[t]{l}100\end{tabular}}}}%
    \put(0.22221339,0.02995104){\makebox(0,0)[lt]{\lineheight{1.25}\smash{\begin{tabular}[t]{l}90\end{tabular}}}}%
    \put(0.31666543,0.02995104){\makebox(0,0)[lt]{\lineheight{1.25}\smash{\begin{tabular}[t]{l}80\end{tabular}}}}%
    \put(0.4131592,0.02995104){\makebox(0,0)[lt]{\lineheight{1.25}\smash{\begin{tabular}[t]{l}70\end{tabular}}}}%
    \put(0.50926282,0.02995104){\makebox(0,0)[lt]{\lineheight{1.25}\smash{\begin{tabular}[t]{l}60\end{tabular}}}}%
    \put(0.60550143,0.02995104){\makebox(0,0)[lt]{\lineheight{1.25}\smash{\begin{tabular}[t]{l}50\end{tabular}}}}%
    \put(0.70179802,0.02995104){\makebox(0,0)[lt]{\lineheight{1.25}\smash{\begin{tabular}[t]{l}40\end{tabular}}}}%
    \put(0.79974639,0.02995104){\makebox(0,0)[lt]{\lineheight{1.25}\smash{\begin{tabular}[t]{l}30\end{tabular}}}}%
    \put(0.89560388,0.03093164){\makebox(0,0)[lt]{\lineheight{1.25}\smash{\begin{tabular}[t]{l}20\end{tabular}}}}%
    \put(0.08080536,0.07879128){\makebox(0,0)[lt]{\lineheight{1.25}\smash{\begin{tabular}[t]{l}4\end{tabular}}}}%
    \put(0.08080536,0.16541814){\makebox(0,0)[lt]{\lineheight{1.25}\smash{\begin{tabular}[t]{l}5\end{tabular}}}}%
    \put(0.08080536,0.25204501){\makebox(0,0)[lt]{\lineheight{1.25}\smash{\begin{tabular}[t]{l}6\end{tabular}}}}%
    \put(0.08080536,0.33867189){\makebox(0,0)[lt]{\lineheight{1.25}\smash{\begin{tabular}[t]{l}7\end{tabular}}}}%
    \put(0.08080536,0.42529874){\makebox(0,0)[lt]{\lineheight{1.25}\smash{\begin{tabular}[t]{l}8\end{tabular}}}}%
    \put(0.08080536,0.51192562){\makebox(0,0)[lt]{\lineheight{1.25}\smash{\begin{tabular}[t]{l}9\end{tabular}}}}%
    \put(0.07009107,0.5985525){\makebox(0,0)[lt]{\lineheight{1.25}\smash{\begin{tabular}[t]{l}10\end{tabular}}}}%
    \put(0.07009107,0.68517937){\makebox(0,0)[lt]{\lineheight{1.25}\smash{\begin{tabular}[t]{l}11\end{tabular}}}}%
  \end{picture}%
\endgroup%

%% file: shortened_data_10mm_mu3.pdf_tex
\begingroup%
  \makeatletter%
  \providecommand\color[2][]{%
    \errmessage{(Inkscape) Color is used for the text in Inkscape, but the package 'color.sty' is not loaded}%
    \renewcommand\color[2][]{}%
  }%
  \providecommand\transparent[1]{%
    \errmessage{(Inkscape) Transparency is used (non-zero) for the text in Inkscape, but the package 'transparent.sty' is not loaded}%
    \renewcommand\transparent[1]{}%
  }%
  \providecommand\rotatebox[2]{#2}%
  \newcommand*\fsize{\dimexpr\f@size pt\relax}%
  \newcommand*\lineheight[1]{\fontsize{\fsize}{#1\fsize}\selectfont}%
  \ifx\svgwidth\undefined%
    \setlength{\unitlength}{420bp}%
    \ifx\svgscale\undefined%
      \relax%
    \else%
      \setlength{\unitlength}{\unitlength * \real{\svgscale}}%
    \fi%
  \else%
    \setlength{\unitlength}{\svgwidth}%
  \fi%
  \global\let\svgwidth\undefined%
  \global\let\svgscale\undefined%
  \makeatother%
  \begin{picture}(1,0.75)%
    \lineheight{1}%
    \setlength\tabcolsep{0pt}%
    \put(0,0){\includegraphics[width=\unitlength,page=1]{shortened_data_10mm_mu3.pdf}}%
    \put(0.31444689,-0.01898537){\color[rgb]{0,0,0}\makebox(0,0)[lt]{\lineheight{1.25}\smash{\begin{tabular}[t]{l}duration time T [s]\end{tabular}}}}%
    \put(0.06901942,0.1930503){\color[rgb]{0,0,0}\rotatebox{90.038036}{\makebox(0,0)[lt]{\lineheight{1.25}\smash{\begin{tabular}[t]{l}stiffness $\mu_3$ [MPa]\end{tabular}}}}}%
    \put(0.12230582,0.03924389){\makebox(0,0)[lt]{\lineheight{1.25}\smash{\begin{tabular}[t]{l}100\end{tabular}}}}%
    \put(0.22227961,0.03924389){\makebox(0,0)[lt]{\lineheight{1.25}\smash{\begin{tabular}[t]{l}90\end{tabular}}}}%
    \put(0.31673164,0.03924389){\makebox(0,0)[lt]{\lineheight{1.25}\smash{\begin{tabular}[t]{l}80\end{tabular}}}}%
    \put(0.41322541,0.03924389){\makebox(0,0)[lt]{\lineheight{1.25}\smash{\begin{tabular}[t]{l}70\end{tabular}}}}%
    \put(0.50932905,0.03924389){\makebox(0,0)[lt]{\lineheight{1.25}\smash{\begin{tabular}[t]{l}60\end{tabular}}}}%
    \put(0.60556766,0.03924389){\makebox(0,0)[lt]{\lineheight{1.25}\smash{\begin{tabular}[t]{l}50\end{tabular}}}}%
    \put(0.70186425,0.03924389){\makebox(0,0)[lt]{\lineheight{1.25}\smash{\begin{tabular}[t]{l}40\end{tabular}}}}%
    \put(0.79981263,0.03924389){\makebox(0,0)[lt]{\lineheight{1.25}\smash{\begin{tabular}[t]{l}30\end{tabular}}}}%
    \put(0.89567011,0.0402245){\makebox(0,0)[lt]{\lineheight{1.25}\smash{\begin{tabular}[t]{l}20\end{tabular}}}}%
    \put(0.09050179,0.34852138){\makebox(0,0)[lt]{\lineheight{1.25}\smash{\begin{tabular}[t]{l}4\end{tabular}}}}%
    \put(0.09050179,0.41663848){\makebox(0,0)[lt]{\lineheight{1.25}\smash{\begin{tabular}[t]{l}5\end{tabular}}}}%
    \put(0.09050179,0.48475563){\makebox(0,0)[lt]{\lineheight{1.25}\smash{\begin{tabular}[t]{l}6\end{tabular}}}}%
    \put(0.09050179,0.55287275){\makebox(0,0)[lt]{\lineheight{1.25}\smash{\begin{tabular}[t]{l}7\end{tabular}}}}%
    \put(0.09050179,0.62098987){\makebox(0,0)[lt]{\lineheight{1.25}\smash{\begin{tabular}[t]{l}8\end{tabular}}}}%
    \put(0.09050179,0.689107){\makebox(0,0)[lt]{\lineheight{1.25}\smash{\begin{tabular}[t]{l}9\end{tabular}}}}%
    \put(0.09100012,0.07769438){\makebox(0,0)[lt]{\lineheight{1.25}\smash{\begin{tabular}[t]{l}0\end{tabular}}}}%
    \put(0.09100012,0.14581154){\makebox(0,0)[lt]{\lineheight{1.25}\smash{\begin{tabular}[t]{l}1\end{tabular}}}}%
    \put(0.09100012,0.21392866){\makebox(0,0)[lt]{\lineheight{1.25}\smash{\begin{tabular}[t]{l}2\end{tabular}}}}%
    \put(0.09100012,0.28204577){\makebox(0,0)[lt]{\lineheight{1.25}\smash{\begin{tabular}[t]{l}3\end{tabular}}}}%
    \put(0,0){\includegraphics[width=\unitlength,page=2]{shortened_data_10mm_mu3.pdf}}%
    \put(0.66109706,0.62560306){\color[rgb]{0,0,0}\makebox(0,0)[lt]{\lineheight{1.25}\smash{\begin{tabular}[t]{l}exact\end{tabular}}}}%
    \put(0.65441568,0.56310268){\color[rgb]{0,0,0}\makebox(0,0)[lt]{\lineheight{1.25}\smash{\begin{tabular}[t]{l}computed\end{tabular}}}}%
  \end{picture}%
\endgroup%

%% file: shortened_data_1mm_mu3.pdf_tex
\begingroup%
  \makeatletter%
  \providecommand\color[2][]{%
    \errmessage{(Inkscape) Color is used for the text in Inkscape, but the package 'color.sty' is not loaded}%
    \renewcommand\color[2][]{}%
  }%
  \providecommand\transparent[1]{%
    \errmessage{(Inkscape) Transparency is used (non-zero) for the text in Inkscape, but the package 'transparent.sty' is not loaded}%
    \renewcommand\transparent[1]{}%
  }%
  \providecommand\rotatebox[2]{#2}%
  \newcommand*\fsize{\dimexpr\f@size pt\relax}%
  \newcommand*\lineheight[1]{\fontsize{\fsize}{#1\fsize}\selectfont}%
  \ifx\svgwidth\undefined%
    \setlength{\unitlength}{420bp}%
    \ifx\svgscale\undefined%
      \relax%
    \else%
      \setlength{\unitlength}{\unitlength * \real{\svgscale}}%
    \fi%
  \else%
    \setlength{\unitlength}{\svgwidth}%
  \fi%
  \global\let\svgwidth\undefined%
  \global\let\svgscale\undefined%
  \makeatother%
  \begin{picture}(1,0.75)%
    \lineheight{1}%
    \setlength\tabcolsep{0pt}%
    \put(0,0){\includegraphics[width=\unitlength,page=1]{shortened_data_1mm_mu3.pdf}}%
    \put(0.24613865,0.62904365){\color[rgb]{0,0,0}\makebox(0,0)[lt]{\lineheight{1.25}\smash{\begin{tabular}[t]{l}exact\end{tabular}}}}%
    \put(0.23945729,0.56654327){\color[rgb]{0,0,0}\makebox(0,0)[lt]{\lineheight{1.25}\smash{\begin{tabular}[t]{l}computed\end{tabular}}}}%
    \put(0.06210562,0.19615486){\color[rgb]{0,0,0}\rotatebox{90.038036}{\makebox(0,0)[lt]{\lineheight{1.25}\smash{\begin{tabular}[t]{l}stiffness $\mu_3$ [MPa]\end{tabular}}}}}%
    \put(0.30359821,-0.02229037){\color[rgb]{0,0,0}\makebox(0,0)[lt]{\lineheight{1.25}\smash{\begin{tabular}[t]{l}duration time T [s]\end{tabular}}}}%
    \put(0.1244277,0.03975409){\makebox(0,0)[lt]{\lineheight{1.25}\smash{\begin{tabular}[t]{l}100\end{tabular}}}}%
    \put(0.2244015,0.03975409){\makebox(0,0)[lt]{\lineheight{1.25}\smash{\begin{tabular}[t]{l}90\end{tabular}}}}%
    \put(0.31885354,0.03975409){\makebox(0,0)[lt]{\lineheight{1.25}\smash{\begin{tabular}[t]{l}80\end{tabular}}}}%
    \put(0.4153473,0.03975409){\makebox(0,0)[lt]{\lineheight{1.25}\smash{\begin{tabular}[t]{l}70\end{tabular}}}}%
    \put(0.51145095,0.03975409){\makebox(0,0)[lt]{\lineheight{1.25}\smash{\begin{tabular}[t]{l}60\end{tabular}}}}%
    \put(0.60768955,0.03975409){\makebox(0,0)[lt]{\lineheight{1.25}\smash{\begin{tabular}[t]{l}50\end{tabular}}}}%
    \put(0.70398614,0.03975409){\makebox(0,0)[lt]{\lineheight{1.25}\smash{\begin{tabular}[t]{l}40\end{tabular}}}}%
    \put(0.80193452,0.03975409){\makebox(0,0)[lt]{\lineheight{1.25}\smash{\begin{tabular}[t]{l}30\end{tabular}}}}%
    \put(0.897792,0.0407347){\makebox(0,0)[lt]{\lineheight{1.25}\smash{\begin{tabular}[t]{l}20\end{tabular}}}}%
    \put(0.08947305,0.48094015){\makebox(0,0)[lt]{\lineheight{1.25}\smash{\begin{tabular}[t]{l}4\end{tabular}}}}%
    \put(0.08947305,0.58228985){\makebox(0,0)[lt]{\lineheight{1.25}\smash{\begin{tabular}[t]{l}5\end{tabular}}}}%
    \put(0.08947305,0.68363962){\makebox(0,0)[lt]{\lineheight{1.25}\smash{\begin{tabular}[t]{l}6\end{tabular}}}}%
    \put(0.08997139,0.07798351){\makebox(0,0)[lt]{\lineheight{1.25}\smash{\begin{tabular}[t]{l}0\end{tabular}}}}%
    \put(0.08997139,0.17933331){\makebox(0,0)[lt]{\lineheight{1.25}\smash{\begin{tabular}[t]{l}1\end{tabular}}}}%
    \put(0.08997139,0.28068305){\makebox(0,0)[lt]{\lineheight{1.25}\smash{\begin{tabular}[t]{l}2\end{tabular}}}}%
    \put(0.08997139,0.38203274){\makebox(0,0)[lt]{\lineheight{1.25}\smash{\begin{tabular}[t]{l}3\end{tabular}}}}%
  \end{picture}%
\endgroup%

%% file: shortened_data_10mm_tau1.pdf_tex
\begingroup%
  \makeatletter%
  \providecommand\color[2][]{%
    \errmessage{(Inkscape) Color is used for the text in Inkscape, but the package 'color.sty' is not loaded}%
    \renewcommand\color[2][]{}%
  }%
  \providecommand\transparent[1]{%
    \errmessage{(Inkscape) Transparency is used (non-zero) for the text in Inkscape, but the package 'transparent.sty' is not loaded}%
    \renewcommand\transparent[1]{}%
  }%
  \providecommand\rotatebox[2]{#2}%
  \newcommand*\fsize{\dimexpr\f@size pt\relax}%
  \newcommand*\lineheight[1]{\fontsize{\fsize}{#1\fsize}\selectfont}%
  \ifx\svgwidth\undefined%
    \setlength{\unitlength}{420bp}%
    \ifx\svgscale\undefined%
      \relax%
    \else%
      \setlength{\unitlength}{\unitlength * \real{\svgscale}}%
    \fi%
  \else%
    \setlength{\unitlength}{\svgwidth}%
  \fi%
  \global\let\svgwidth\undefined%
  \global\let\svgscale\undefined%
  \makeatother%
  \begin{picture}(1,0.75)%
    \lineheight{1}%
    \setlength\tabcolsep{0pt}%
    \put(0,0){\includegraphics[width=\unitlength,page=1]{shortened_data_10mm_tau1.pdf}}%
    \put(0.3052433,-0.02234775){\color[rgb]{0,0,0}\makebox(0,0)[lt]{\lineheight{1.25}\smash{\begin{tabular}[t]{l}duration time T [s]\end{tabular}}}}%
    \put(0.12413717,0.03835104){\makebox(0,0)[lt]{\lineheight{1.25}\smash{\begin{tabular}[t]{l}100\end{tabular}}}}%
    \put(0.22411097,0.03835104){\makebox(0,0)[lt]{\lineheight{1.25}\smash{\begin{tabular}[t]{l}90\end{tabular}}}}%
    \put(0.31856302,0.03835104){\makebox(0,0)[lt]{\lineheight{1.25}\smash{\begin{tabular}[t]{l}80\end{tabular}}}}%
    \put(0.41505679,0.03835104){\makebox(0,0)[lt]{\lineheight{1.25}\smash{\begin{tabular}[t]{l}70\end{tabular}}}}%
    \put(0.51116043,0.03835104){\makebox(0,0)[lt]{\lineheight{1.25}\smash{\begin{tabular}[t]{l}60\end{tabular}}}}%
    \put(0.60739904,0.03835104){\makebox(0,0)[lt]{\lineheight{1.25}\smash{\begin{tabular}[t]{l}50\end{tabular}}}}%
    \put(0.70369563,0.03835104){\makebox(0,0)[lt]{\lineheight{1.25}\smash{\begin{tabular}[t]{l}40\end{tabular}}}}%
    \put(0.801644,0.03835104){\makebox(0,0)[lt]{\lineheight{1.25}\smash{\begin{tabular}[t]{l}30\end{tabular}}}}%
    \put(0.89750148,0.03933164){\makebox(0,0)[lt]{\lineheight{1.25}\smash{\begin{tabular}[t]{l}20\end{tabular}}}}%
    \put(0.010753,0.19514356){\color[rgb]{0,0,0}\rotatebox{90.038036}{\makebox(0,0)[lt]{\lineheight{1.25}\smash{\begin{tabular}[t]{l}relaxation time $\tau_1$ [s]\end{tabular}}}}}%
    \put(0.03675,0.35040859){\makebox(0,0)[lt]{\lineheight{1.25}\smash{\begin{tabular}[t]{l}0.17\end{tabular}}}}%
    \put(0.03675,0.4185257){\makebox(0,0)[lt]{\lineheight{1.25}\smash{\begin{tabular}[t]{l}0.18\end{tabular}}}}%
    \put(0.03675,0.48664282){\makebox(0,0)[lt]{\lineheight{1.25}\smash{\begin{tabular}[t]{l}0.19\end{tabular}}}}%
    \put(0.03675,0.55475993){\makebox(0,0)[lt]{\lineheight{1.25}\smash{\begin{tabular}[t]{l}0.2\end{tabular}}}}%
    \put(0.03675,0.62287704){\makebox(0,0)[lt]{\lineheight{1.25}\smash{\begin{tabular}[t]{l}0.21\end{tabular}}}}%
    \put(0.03724833,0.07958159){\makebox(0,0)[lt]{\lineheight{1.25}\smash{\begin{tabular}[t]{l}0.13\end{tabular}}}}%
    \put(0.03724833,0.14769875){\makebox(0,0)[lt]{\lineheight{1.25}\smash{\begin{tabular}[t]{l}0.14\end{tabular}}}}%
    \put(0.03724833,0.21581588){\makebox(0,0)[lt]{\lineheight{1.25}\smash{\begin{tabular}[t]{l}0.15\end{tabular}}}}%
    \put(0.03724833,0.28393298){\makebox(0,0)[lt]{\lineheight{1.25}\smash{\begin{tabular}[t]{l}0.16\end{tabular}}}}%
    \put(0,0){\includegraphics[width=\unitlength,page=2]{shortened_data_10mm_tau1.pdf}}%
    \put(0.66962422,0.63501282){\color[rgb]{0,0,0}\makebox(0,0)[lt]{\lineheight{1.25}\smash{\begin{tabular}[t]{l}exact\end{tabular}}}}%
    \put(0.66294289,0.57251242){\color[rgb]{0,0,0}\makebox(0,0)[lt]{\lineheight{1.25}\smash{\begin{tabular}[t]{l}computed\end{tabular}}}}%
  \end{picture}%
\endgroup%

%% file: shortened_data_1mm_tau1.pdf_tex
\begingroup%
  \makeatletter%
  \providecommand\color[2][]{%
    \errmessage{(Inkscape) Color is used for the text in Inkscape, but the package 'color.sty' is not loaded}%
    \renewcommand\color[2][]{}%
  }%
  \providecommand\transparent[1]{%
    \errmessage{(Inkscape) Transparency is used (non-zero) for the text in Inkscape, but the package 'transparent.sty' is not loaded}%
    \renewcommand\transparent[1]{}%
  }%
  \providecommand\rotatebox[2]{#2}%
  \newcommand*\fsize{\dimexpr\f@size pt\relax}%
  \newcommand*\lineheight[1]{\fontsize{\fsize}{#1\fsize}\selectfont}%
  \ifx\svgwidth\undefined%
    \setlength{\unitlength}{420bp}%
    \ifx\svgscale\undefined%
      \relax%
    \else%
      \setlength{\unitlength}{\unitlength * \real{\svgscale}}%
    \fi%
  \else%
    \setlength{\unitlength}{\svgwidth}%
  \fi%
  \global\let\svgwidth\undefined%
  \global\let\svgscale\undefined%
  \makeatother%
  \begin{picture}(1,0.75)%
    \lineheight{1}%
    \setlength\tabcolsep{0pt}%
    \put(0,0){\includegraphics[width=\unitlength,page=1]{shortened_data_1mm_tau1.pdf}}%
    \put(0.21527442,0.63540904){\color[rgb]{0,0,0}\makebox(0,0)[lt]{\lineheight{1.25}\smash{\begin{tabular}[t]{l}exact\end{tabular}}}}%
    \put(0.2085931,0.57290866){\color[rgb]{0,0,0}\makebox(0,0)[lt]{\lineheight{1.25}\smash{\begin{tabular}[t]{l}computed\end{tabular}}}}%
    \put(0.03234522,0.19522715){\color[rgb]{0,0,0}\rotatebox{90.038036}{\makebox(0,0)[lt]{\lineheight{1.25}\smash{\begin{tabular}[t]{l}relaxation time $\tau_1$ [s]\end{tabular}}}}}%
    \put(0.30528079,-0.02219772){\color[rgb]{0,0,0}\makebox(0,0)[lt]{\lineheight{1.25}\smash{\begin{tabular}[t]{l}duration time T [s]\end{tabular}}}}%
    \put(0.12373895,0.03832604){\makebox(0,0)[lt]{\lineheight{1.25}\smash{\begin{tabular}[t]{l}100\end{tabular}}}}%
    \put(0.22371275,0.03832604){\makebox(0,0)[lt]{\lineheight{1.25}\smash{\begin{tabular}[t]{l}90\end{tabular}}}}%
    \put(0.3181648,0.03832604){\makebox(0,0)[lt]{\lineheight{1.25}\smash{\begin{tabular}[t]{l}80\end{tabular}}}}%
    \put(0.41465857,0.03832604){\makebox(0,0)[lt]{\lineheight{1.25}\smash{\begin{tabular}[t]{l}70\end{tabular}}}}%
    \put(0.51076222,0.03832604){\makebox(0,0)[lt]{\lineheight{1.25}\smash{\begin{tabular}[t]{l}60\end{tabular}}}}%
    \put(0.60700082,0.03832604){\makebox(0,0)[lt]{\lineheight{1.25}\smash{\begin{tabular}[t]{l}50\end{tabular}}}}%
    \put(0.70329741,0.03832604){\makebox(0,0)[lt]{\lineheight{1.25}\smash{\begin{tabular}[t]{l}40\end{tabular}}}}%
    \put(0.80124579,0.03832604){\makebox(0,0)[lt]{\lineheight{1.25}\smash{\begin{tabular}[t]{l}30\end{tabular}}}}%
    \put(0.89710327,0.03930664){\makebox(0,0)[lt]{\lineheight{1.25}\smash{\begin{tabular}[t]{l}20\end{tabular}}}}%
    \put(0.05817857,0.38209477){\makebox(0,0)[lt]{\lineheight{1.25}\smash{\begin{tabular}[t]{l}0.5\end{tabular}}}}%
    \put(0.05817857,0.4581764){\makebox(0,0)[lt]{\lineheight{1.25}\smash{\begin{tabular}[t]{l}0.6\end{tabular}}}}%
    \put(0.05817857,0.53425804){\makebox(0,0)[lt]{\lineheight{1.25}\smash{\begin{tabular}[t]{l}0.7\end{tabular}}}}%
    \put(0.05817857,0.61033967){\makebox(0,0)[lt]{\lineheight{1.25}\smash{\begin{tabular}[t]{l}0.8\end{tabular}}}}%
    \put(0.05817857,0.68642131){\makebox(0,0)[lt]{\lineheight{1.25}\smash{\begin{tabular}[t]{l}0.9\end{tabular}}}}%
    \put(0.05867691,0.07960165){\makebox(0,0)[lt]{\lineheight{1.25}\smash{\begin{tabular}[t]{l}0.1\end{tabular}}}}%
    \put(0.05867691,0.15568333){\makebox(0,0)[lt]{\lineheight{1.25}\smash{\begin{tabular}[t]{l}0.2\end{tabular}}}}%
    \put(0.05867691,0.23176497){\makebox(0,0)[lt]{\lineheight{1.25}\smash{\begin{tabular}[t]{l}0.3\end{tabular}}}}%
    \put(0.05867691,0.30784659){\makebox(0,0)[lt]{\lineheight{1.25}\smash{\begin{tabular}[t]{l}0.4\end{tabular}}}}%
  \end{picture}%
\endgroup%

%% file: shortened_data_10mm_tau3.pdf_tex
\begingroup%
  \makeatletter%
  \providecommand\color[2][]{%
    \errmessage{(Inkscape) Color is used for the text in Inkscape, but the package 'color.sty' is not loaded}%
    \renewcommand\color[2][]{}%
  }%
  \providecommand\transparent[1]{%
    \errmessage{(Inkscape) Transparency is used (non-zero) for the text in Inkscape, but the package 'transparent.sty' is not loaded}%
    \renewcommand\transparent[1]{}%
  }%
  \providecommand\rotatebox[2]{#2}%
  \newcommand*\fsize{\dimexpr\f@size pt\relax}%
  \newcommand*\lineheight[1]{\fontsize{\fsize}{#1\fsize}\selectfont}%
  \ifx\svgwidth\undefined%
    \setlength{\unitlength}{420bp}%
    \ifx\svgscale\undefined%
      \relax%
    \else%
      \setlength{\unitlength}{\unitlength * \real{\svgscale}}%
    \fi%
  \else%
    \setlength{\unitlength}{\svgwidth}%
  \fi%
  \global\let\svgwidth\undefined%
  \global\let\svgscale\undefined%
  \makeatother%
  \begin{picture}(1,0.75)%
    \lineheight{1}%
    \setlength\tabcolsep{0pt}%
    \put(0,0){\includegraphics[width=\unitlength,page=1]{shortened_data_10mm_tau3.pdf}}%
    \put(0.66883371,0.17254186){\color[rgb]{0,0,0}\makebox(0,0)[lt]{\lineheight{1.25}\smash{\begin{tabular}[t]{l}exact\end{tabular}}}}%
    \put(0.66215237,0.11004148){\color[rgb]{0,0,0}\makebox(0,0)[lt]{\lineheight{1.25}\smash{\begin{tabular}[t]{l}computed\end{tabular}}}}%
    \put(0.03234523,0.19522693){\color[rgb]{0,0,0}\rotatebox{90.038036}{\makebox(0,0)[lt]{\lineheight{1.25}\smash{\begin{tabular}[t]{l}relaxation time $\tau_3$ [s]\end{tabular}}}}}%
    \put(0.30528079,-0.02219772){\color[rgb]{0,0,0}\makebox(0,0)[lt]{\lineheight{1.25}\smash{\begin{tabular}[t]{l}duration time T [s]\end{tabular}}}}%
    \put(0.12373895,0.03832604){\makebox(0,0)[lt]{\lineheight{1.25}\smash{\begin{tabular}[t]{l}100\end{tabular}}}}%
    \put(0.22371275,0.03832604){\makebox(0,0)[lt]{\lineheight{1.25}\smash{\begin{tabular}[t]{l}90\end{tabular}}}}%
    \put(0.3181648,0.03832604){\makebox(0,0)[lt]{\lineheight{1.25}\smash{\begin{tabular}[t]{l}80\end{tabular}}}}%
    \put(0.41465857,0.03832604){\makebox(0,0)[lt]{\lineheight{1.25}\smash{\begin{tabular}[t]{l}70\end{tabular}}}}%
    \put(0.51076222,0.03832604){\makebox(0,0)[lt]{\lineheight{1.25}\smash{\begin{tabular}[t]{l}60\end{tabular}}}}%
    \put(0.60700084,0.03832604){\makebox(0,0)[lt]{\lineheight{1.25}\smash{\begin{tabular}[t]{l}50\end{tabular}}}}%
    \put(0.70329743,0.03832604){\makebox(0,0)[lt]{\lineheight{1.25}\smash{\begin{tabular}[t]{l}40\end{tabular}}}}%
    \put(0.8012458,0.03832604){\makebox(0,0)[lt]{\lineheight{1.25}\smash{\begin{tabular}[t]{l}30\end{tabular}}}}%
    \put(0.89710329,0.03930664){\makebox(0,0)[lt]{\lineheight{1.25}\smash{\begin{tabular}[t]{l}20\end{tabular}}}}%
    \put(0.05817857,0.48206571){\makebox(0,0)[lt]{\lineheight{1.25}\smash{\begin{tabular}[t]{l}30\end{tabular}}}}%
    \put(0.05817857,0.68640709){\makebox(0,0)[lt]{\lineheight{1.25}\smash{\begin{tabular}[t]{l}35\end{tabular}}}}%
    \put(0.05867691,0.07584508){\makebox(0,0)[lt]{\lineheight{1.25}\smash{\begin{tabular}[t]{l}20\end{tabular}}}}%
    \put(0.05867691,0.28018647){\makebox(0,0)[lt]{\lineheight{1.25}\smash{\begin{tabular}[t]{l}25\end{tabular}}}}%
  \end{picture}%
\endgroup%

%% file: shortened_data_1mm_tau3.pdf_tex
\begingroup%
  \makeatletter%
  \providecommand\color[2][]{%
    \errmessage{(Inkscape) Color is used for the text in Inkscape, but the package 'color.sty' is not loaded}%
    \renewcommand\color[2][]{}%
  }%
  \providecommand\transparent[1]{%
    \errmessage{(Inkscape) Transparency is used (non-zero) for the text in Inkscape, but the package 'transparent.sty' is not loaded}%
    \renewcommand\transparent[1]{}%
  }%
  \providecommand\rotatebox[2]{#2}%
  \newcommand*\fsize{\dimexpr\f@size pt\relax}%
  \newcommand*\lineheight[1]{\fontsize{\fsize}{#1\fsize}\selectfont}%
  \ifx\svgwidth\undefined%
    \setlength{\unitlength}{420bp}%
    \ifx\svgscale\undefined%
      \relax%
    \else%
      \setlength{\unitlength}{\unitlength * \real{\svgscale}}%
    \fi%
  \else%
    \setlength{\unitlength}{\svgwidth}%
  \fi%
  \global\let\svgwidth\undefined%
  \global\let\svgscale\undefined%
  \makeatother%
  \begin{picture}(1,0.75)%
    \lineheight{1}%
    \setlength\tabcolsep{0pt}%
    \put(0,0){\includegraphics[width=\unitlength,page=1]{shortened_data_1mm_tau3.pdf}}%
    \put(0.30528079,-0.02219772){\color[rgb]{0,0,0}\makebox(0,0)[lt]{\lineheight{1.25}\smash{\begin{tabular}[t]{l}duration time T [s]\end{tabular}}}}%
    \put(0.03234522,0.19522693){\color[rgb]{0,0,0}\rotatebox{90.038036}{\makebox(0,0)[lt]{\lineheight{1.25}\smash{\begin{tabular}[t]{l}relaxation time $\tau_3$ [s]\end{tabular}}}}}%
    \put(0.05839391,0.4826536){\makebox(0,0)[lt]{\lineheight{1.25}\smash{\begin{tabular}[t]{l}30\end{tabular}}}}%
    \put(0.05839391,0.68699499){\makebox(0,0)[lt]{\lineheight{1.25}\smash{\begin{tabular}[t]{l}35\end{tabular}}}}%
    \put(0.05889224,0.07643297){\makebox(0,0)[lt]{\lineheight{1.25}\smash{\begin{tabular}[t]{l}20\end{tabular}}}}%
    \put(0.05889224,0.28077437){\makebox(0,0)[lt]{\lineheight{1.25}\smash{\begin{tabular}[t]{l}25\end{tabular}}}}%
    \put(0.12373895,0.03832604){\makebox(0,0)[lt]{\lineheight{1.25}\smash{\begin{tabular}[t]{l}100\end{tabular}}}}%
    \put(0.22371275,0.03832604){\makebox(0,0)[lt]{\lineheight{1.25}\smash{\begin{tabular}[t]{l}90\end{tabular}}}}%
    \put(0.3181648,0.03832604){\makebox(0,0)[lt]{\lineheight{1.25}\smash{\begin{tabular}[t]{l}80\end{tabular}}}}%
    \put(0.41465857,0.03832604){\makebox(0,0)[lt]{\lineheight{1.25}\smash{\begin{tabular}[t]{l}70\end{tabular}}}}%
    \put(0.51076222,0.03832604){\makebox(0,0)[lt]{\lineheight{1.25}\smash{\begin{tabular}[t]{l}60\end{tabular}}}}%
    \put(0.60700084,0.03832604){\makebox(0,0)[lt]{\lineheight{1.25}\smash{\begin{tabular}[t]{l}50\end{tabular}}}}%
    \put(0.70329743,0.03832604){\makebox(0,0)[lt]{\lineheight{1.25}\smash{\begin{tabular}[t]{l}40\end{tabular}}}}%
    \put(0.8012458,0.03832604){\makebox(0,0)[lt]{\lineheight{1.25}\smash{\begin{tabular}[t]{l}30\end{tabular}}}}%
    \put(0.89710329,0.03930664){\makebox(0,0)[lt]{\lineheight{1.25}\smash{\begin{tabular}[t]{l}20\end{tabular}}}}%
    \put(0,0){\includegraphics[width=\unitlength,page=2]{shortened_data_1mm_tau3.pdf}}%
    \put(0.21648026,0.17492371){\color[rgb]{0,0,0}\makebox(0,0)[lt]{\lineheight{1.25}\smash{\begin{tabular}[t]{l}exact\end{tabular}}}}%
    \put(0.20979892,0.11242332){\color[rgb]{0,0,0}\makebox(0,0)[lt]{\lineheight{1.25}\smash{\begin{tabular}[t]{l}computed\end{tabular}}}}%
  \end{picture}%
\endgroup%

%% file: stress_reconstruction_time.pdf_tex
\begingroup%
  \makeatletter%
  \providecommand\color[2][]{%
    \errmessage{(Inkscape) Color is used for the text in Inkscape, but the package 'color.sty' is not loaded}%
    \renewcommand\color[2][]{}%
  }%
  \providecommand\transparent[1]{%
    \errmessage{(Inkscape) Transparency is used (non-zero) for the text in Inkscape, but the package 'transparent.sty' is not loaded}%
    \renewcommand\transparent[1]{}%
  }%
  \providecommand\rotatebox[2]{#2}%
  \newcommand*\fsize{\dimexpr\f@size pt\relax}%
  \newcommand*\lineheight[1]{\fontsize{\fsize}{#1\fsize}\selectfont}%
  \ifx\svgwidth\undefined%
    \setlength{\unitlength}{420bp}%
    \ifx\svgscale\undefined%
      \relax%
    \else%
      \setlength{\unitlength}{\unitlength * \real{\svgscale}}%
    \fi%
  \else%
    \setlength{\unitlength}{\svgwidth}%
  \fi%
  \global\let\svgwidth\undefined%
  \global\let\svgscale\undefined%
  \makeatother%
  \begin{picture}(1,0.75)%
    \lineheight{1}%
    \setlength\tabcolsep{0pt}%
    \put(0,0){\includegraphics[width=\unitlength,page=1]{stress_reconstruction_time.pdf}}%
    \put(0.05561296,0.41489971){\rotatebox{90}{\makebox(0,0)[t]{\lineheight{1.25}\smash{\begin{tabular}[t]{c}Stress [MPa]\end{tabular}}}}}%
    \put(0.5109002,0.01406139){\makebox(0,0)[t]{\lineheight{1.25}\smash{\begin{tabular}[t]{c}Time [s]\end{tabular}}}}%
    \put(0.81494837,0.66586968){\makebox(0,0)[lt]{\lineheight{1.25}\smash{\begin{tabular}[t]{l}1mm/s\end{tabular}}}}%
    \put(0.81494837,0.63297796){\makebox(0,0)[lt]{\lineheight{1.25}\smash{\begin{tabular}[t]{l}10mm/s\end{tabular}}}}%
    \put(0.13098724,0.04484054){\makebox(0,0)[t]{\lineheight{1.25}\smash{\begin{tabular}[t]{c}0\end{tabular}}}}%
    \put(0.28425664,0.04482891){\makebox(0,0)[t]{\lineheight{1.25}\smash{\begin{tabular}[t]{c}20\end{tabular}}}}%
    \put(0.43696675,0.04484054){\makebox(0,0)[t]{\lineheight{1.25}\smash{\begin{tabular}[t]{c}40\end{tabular}}}}%
    \put(0.59215491,0.04484054){\makebox(0,0)[t]{\lineheight{1.25}\smash{\begin{tabular}[t]{c}60\end{tabular}}}}%
    \put(0.7485822,0.04479404){\makebox(0,0)[t]{\lineheight{1.25}\smash{\begin{tabular}[t]{c}80\end{tabular}}}}%
    \put(0.90377054,0.04484054){\makebox(0,0)[t]{\lineheight{1.25}\smash{\begin{tabular}[t]{c}100\end{tabular}}}}%
    \put(0.18586825,0.04484051){\color[rgb]{0,0,0}\makebox(0,0)[lt]{\lineheight{1.25}\smash{\begin{tabular}[t]{l}10\end{tabular}}}}%
    \put(0.34571495,0.04482891){\color[rgb]{0,0,0}\makebox(0,0)[lt]{\lineheight{1.25}\smash{\begin{tabular}[t]{l}30\end{tabular}}}}%
    \put(0.50184424,0.04484051){\color[rgb]{0,0,0}\makebox(0,0)[lt]{\lineheight{1.25}\smash{\begin{tabular}[t]{l}50\end{tabular}}}}%
    \put(0.65797364,0.04484051){\color[rgb]{0,0,0}\makebox(0,0)[lt]{\lineheight{1.25}\smash{\begin{tabular}[t]{l}70\end{tabular}}}}%
    \put(0.81410294,0.04484051){\color[rgb]{0,0,0}\makebox(0,0)[lt]{\lineheight{1.25}\smash{\begin{tabular}[t]{l}90\end{tabular}}}}%
    \put(0.11722345,0.07327113){\makebox(0,0)[rt]{\lineheight{1.25}\smash{\begin{tabular}[t]{r}0\end{tabular}}}}%
    \put(0.11722345,0.16077113){\makebox(0,0)[rt]{\lineheight{1.25}\smash{\begin{tabular}[t]{r}50\end{tabular}}}}%
    \put(0.11722345,0.24827113){\makebox(0,0)[rt]{\lineheight{1.25}\smash{\begin{tabular}[t]{r}100\end{tabular}}}}%
    \put(0.11722345,0.33577113){\makebox(0,0)[rt]{\lineheight{1.25}\smash{\begin{tabular}[t]{r}150\end{tabular}}}}%
    \put(0.11722345,0.42327113){\makebox(0,0)[rt]{\lineheight{1.25}\smash{\begin{tabular}[t]{r}200\end{tabular}}}}%
    \put(0.11722345,0.51077113){\makebox(0,0)[rt]{\lineheight{1.25}\smash{\begin{tabular}[t]{r}250\end{tabular}}}}%
    \put(0.11722345,0.59827114){\makebox(0,0)[rt]{\lineheight{1.25}\smash{\begin{tabular}[t]{r}300\end{tabular}}}}%
    \put(0.11722345,0.68577114){\makebox(0,0)[rt]{\lineheight{1.25}\smash{\begin{tabular}[t]{r}350\end{tabular}}}}%
    \put(0.54490912,0.4727493){\color[rgb]{1,0,0}\makebox(0,0)[lt]{\lineheight{1.25}\smash{\begin{tabular}[t]{l}$\mu$\end{tabular}}}}%
    \put(0.50355513,0.4727493){\color[rgb]{1,0,0}\makebox(0,0)[lt]{\lineheight{1.25}\smash{\begin{tabular}[t]{l}$\mu_3$\end{tabular}}}}%
    \put(0.89272292,0.4727493){\color[rgb]{1,0,0}\makebox(0,0)[lt]{\lineheight{1.25}\smash{\begin{tabular}[t]{l}$\tau_1$\end{tabular}}}}%
    \put(0.46345068,0.4727493){\color[rgb]{1,0,0}\makebox(0,0)[lt]{\lineheight{1.25}\smash{\begin{tabular}[t]{l}$\tau_3$\end{tabular}}}}%
    \put(0.31430868,0.4727493){\color[rgb]{0,0,1}\makebox(0,0)[lt]{\lineheight{1.25}\smash{\begin{tabular}[t]{l}$\mu$\end{tabular}}}}%
    \put(0.42191764,0.4727493){\color[rgb]{0,0,1}\makebox(0,0)[lt]{\lineheight{1.25}\smash{\begin{tabular}[t]{l}$\mu_3$\end{tabular}}}}%
    \put(0.50355513,0.49739955){\color[rgb]{0,0,1}\makebox(0,0)[lt]{\lineheight{1.25}\smash{\begin{tabular}[t]{l}$\tau_1$\end{tabular}}}}%
    \put(0.65971533,0.4727493){\color[rgb]{0,0,1}\makebox(0,0)[lt]{\lineheight{1.25}\smash{\begin{tabular}[t]{l}$\tau_3$\end{tabular}}}}%
  \end{picture}%
\endgroup%